\documentclass{article}
\usepackage{tikz}
\usepackage{amsmath,amssymb,amsthm}

\def\bfH{{\mathbf{H}}}
\def\bfe{{\mathbf{e}}}

\def\bbC{{\mathbb{C}}}
\def\bbN{{\mathbb{N}}}
\def\bbR{{\mathbb{R}}}
\def\bbZ{{\mathbb{Z}}}
\def\calC{{\mathcal{C}}}
\def\calE{{\mathcal{E}}}
\def\calM{{\mathcal{M}}}
\def\calS{{\mathcal{S}}}
\def\eps{{\varepsilon}}
\def\micron{{~\mu\mathrm{m}}}
\def\cminv{{~\mathrm{cm}^{-1}}}
\def\radminv{{~\mathrm{rad}.\mathrm{m}^{-1}}}
\def\sinc{\operatorname{sinc}}

\def\txi{{\tilde\xi}}
\newtheorem{theorem}{Theorem}
\newtheorem{proposition}{Proposition}

\begin{document}
\title{Impact of the Metallic Interface Description on Sub-wavelength Cavity Mode Computations}
\author{Brigitte Bidegaray-Fesquet\footnote{Laboratoire Jean Kuntzmann,
CNRS, Universit\'e Joseph Fourier, Grenoble INP, Universit\'e Pierre Mend\`es-France, BP 53, 38 041 Grenoble Cedex 9, France, \texttt{Brigitte.Bidegaray@imag.fr}. Corresponding author.}, 
\'Eric Dumas\footnote{Institut Fourier,
Universit\'e Joseph Fourier, CNRS, BP 74, 100 rue des Math\'ematiques, 38402 Saint Martin d'H\`eres, France, \texttt{edumas@ujf-grenoble.fr}}}
\maketitle

\begin{abstract}
We present a numerical study of electromagnetic reflection and cavity modes of 1D-sub-wavelength rectangular metallic gratings exposed to TM-polarized light. 
Computations are made using the modal development. 
In particular we study the influence of the choice of boundary conditions on the metallic surfaces on the determination of modes, on specular reflectance and cavity mode amplitudes. 
Our full real-metal approach shows some advantages when compared to former results since it is in better accordance with experimental results.
\end{abstract}

AMS: 30E15, 35B30, 35J05, 35Q60, 78A10 \\
Keywords: Helmholtz, sub-wavelength optics, boundary condition, impedance condition, asymptotics.

\maketitle

\section{Introduction}

The discovery of Surface Enhanced Raman Scattering has reinforced the interest in the study of the optical properties of rough metallic surfaces, in particular the understanding of near-field features.

We consider the reflection of an electromagnetic wave on the grating described on Figure \ref{fig.geom}. 
This grating is supposed to be infinite in the $z$-direction and periodic in the $x$-direction, with period $d$. 
The grating consists of grooves of width $w$ and depth $h$. 
The width $w$ (and in practice the period $d$, see \cite{Barbara-Quemerais-Bustarret-LopezRios-Fournier03} and physical parameters, Section \ref{sec.param}) is supposed to be smaller than the wavelength. 

\begin{figure}[htbp]
\begin{center}
\input{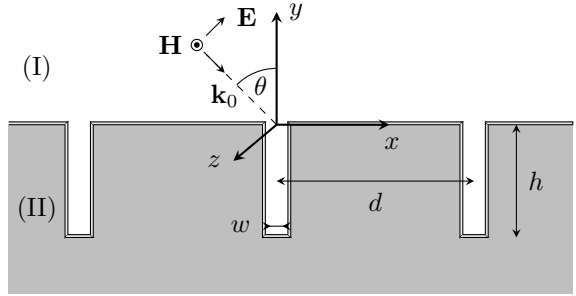}
\caption{\label{fig.geom}Reflecting grating in TM-polarization}
\end{center}
\end{figure}

The dimensions of this grating give rise to two effects when exposed to TM-polarized light.
First, since its period is of the order of the optical wavelength, experiments show extinctions of the specular reflectivity at specific incident angles. 
These extinctions stem from the transfer of energy to surface plasmons, which are very intense and propagate along the surface. 
Besides the studied gratings have also a depth of the order of the wavelength and this yields resonances inside the cavity.

Given a monochromatic incident wave with a wave vector of modulus $k_0$, we want to compute the reflected modes in the domain (I) and the cavity modes in the cavities (II), such that the resulting magnetic field $\bfH=H_z\bfe_z$ is solution to the Helmholtz equation
\begin{equation}
\label{Helmholtz}
\Delta_{x,y} H_z(x,y) + k_0^2 H_z(x,y) = 0.
\end{equation}

We apply here the modal development method \cite{Sheng-Stepleman-Sanda82}, which consists in writing the field in (I) and (II) as a sum of modes and using continuity properties to link both expansions. 
This method has the advantage to lead to very simple and efficient computations. 
Other more elaborate methods have also been successfully applied to diffraction gratings, such as integral methods, coupling of finite elements and Rayleigh expansion,\dots\ (see e.g. \cite{Bao-Cowsar-Masters01}).
This computation requires a choice for the treatment of the surface which is covered with gold.

Modal computations can be easily made treating all the gold surfaces with a perfectly conducting boundary condition (perfect metal).
In \cite{Barbara-Quemerais-Bustarret-LopezRios-Fournier03}, walls (vertical surfaces) are treated with such a condition and horizontal surfaces with a surface impedance condition (real metal).
Our aim here is to compare these computations with full real-metal ones and to analyze the impact of the surface condition choice on the computed mode coefficients.

In Section \ref{sec.setting} we describe the modal development and the different boundary conditions addressed.
We give the general outline of the modal method and the physical parameters used for the simulations which reproduce experiments. 
In Section \ref{sec.modes}, we explain how to find the cavity modes, perform an asymptotic expansion of the $n$th modal wave number in terms of $n$ and of the surface impedance, and study cavity mode energies.
We also show that these modes span the functional space of interest, which justifies \textit{a posteriori} the use of the modal method, even for real-metal surface conditions.
In Section \ref{sec.num}, we give hints on how to tune numerical parameters to obtain qualitatively good results compared to specular reflectance experimental measures.
On the way, we compare \cite{Barbara-Quemerais-Bustarret-LopezRios-Fournier03}-type computations and the full real-metal case and end up with cavity mode results.  

\section{Problem setting}
\label{sec.setting}

\subsection{Reflected and cavity modes}

Given an incident wave $H_z^i=e^{ik_0(\sin\theta x-\cos\theta y)}$ with a wave vector of modulus $k_0$ and an angle $\theta$ with respect to the vertical direction (see Figure \ref{fig.geom}), the reflected waves in domain (I) are only those which stem from constructive interferences and are given by a Rayleigh expansion
\begin{equation}
\label{HI}
H^I_z(x,y) = e^{ik_0(\gamma_0x-\beta_0y)}
+ \sum_{m\in\bbZ} R_m e^{ik_0(\gamma_mx+\beta_my)},
\end{equation}
where $\gamma_m=\sin\theta+2\pi m/k_0d$ and $\beta_m^2+\gamma_m^2=1$ ($\beta_m$ is chosen to have a non-negative imaginary part). 
In particular, $H_z^I$ is $d$-quasi-periodic: $H^I_z(x+d,y) = e^{ik_0\gamma_0d}H^I_z(x,y)$. 
Only a finite number of these waves, corresponding to a real $\beta_m$, are propagative.
One of our goals is to determine the reflection coefficients $R_m$, $m\in\bbZ$.
The second goal is to determine the cavity modes.
We seek them in the Sheng \cite{Sheng-Stepleman-Sanda82} mode decomposition, i.e. $H^{II}_z(x,y)$ is written as a superposition of waves of the type
\begin{equation}
\label{HII}
(Ae^{ik_0\mu x}+Be^{-ik_0\mu x}) (Ce^{ik_0\nu y}+De^{-ik_0\nu y}),
\end{equation}
where $\mu$ and $\nu$ are solution to $\mu^2+\nu^2=1$.

\subsection{Boundary conditions}

Impedance conditions are used to model the interface between a dielectric medium and a metal.
Indeed we do not want to model the electromagnetic wave which penetrates in the metal within the skin depth, but replace it by the condition
\begin{equation}
\label{impedance}
\partial_n H_z + ik_0 Z H_z = 0,
\end{equation}
where $n$ is the exterior normal of the metallic surface and $Z=\sqrt{\eps_1/\eps_m}$ (for the usual determination of the square root) is the relative surface impedance, computed from the relative dielectric constants $\eps_1$ and $\eps_m$ of the dielectric medium and the metal respectively. 
Here the dielectric medium is air and we denote $Z=\xi=1/\sqrt{\eps_m}$.
For a real metal the dielectric constant is complex and consists of a large negative real part and a smaller imaginary part. Numerical values of $\xi$ stemming from \cite{Palik85} are given on Table \ref{tab.xi}.
Considering a perfect metal interface consists in letting $\eps_m\to-\infty$, which yields $\xi=0$ and a Neumann boundary condition $\partial_n H_z = 0$.

We compare the following assumptions on the metallic surfaces (see also Figure \ref{fig.cases}):
\begin{description}
\itemsep-1mm
\item[(P) case:] perfect metal (Neumann condition) everywhere;
\item[(M$_0$) case:] perfect metal on walls and real metal on horizontal surfaces neglecting $\Im(\eps_m)$;
\item[(M) case:] same as (M$_0$) but taking into account $\Im(\eps_m)$ (mixed case of \cite{Barbara-Quemerais-Bustarret-LopezRios-Fournier03});
\item[(R$_0$) case:] real metal (impedance condition) everywhere neglecting $\Im(\eps_m)$;
\item[(R) case:] same as (R$_0$) but taking into account $\Im(\eps_m)$.
\end{description} 

\begin{figure}[htbp]
\begin{center}
\input{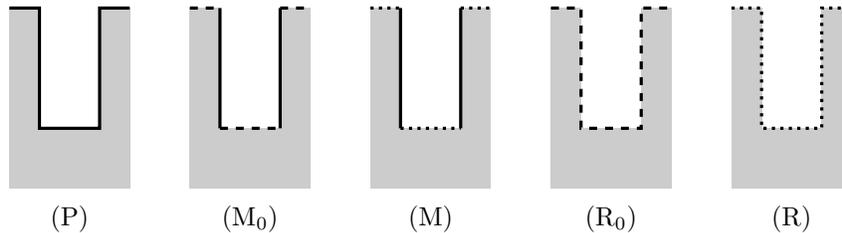}
\caption{\label{fig.cases}Different assumptions on the metallic surface: perfect metal (plain line), real metal neglecting $\Im(\eps_m)$ (dashed line), real metal (dotted line).}
\end{center}
\end{figure}

\subsection{Methodology}
\label{Sec.method}

We follow the computations in \cite{Barbara-Quemerais-Bustarret-LopezRios-Fournier03} which consist in three steps: 
\begin{description}
\itemsep-1mm
\item[step 1.] take into account the boundary condition \eqref{impedance} on the vertical walls.
This defines a relation between $A$ and $B$ in \eqref{HII}, and discrete values for $\mu$ which we denote by $\mu_n$, $n\in\bbN$. 
The corresponding values $\nu_n$ of $\nu$ follow immediately.
\item[step 2.] take into account the boundary condition \eqref{impedance} on the bottom of the cavity.
This defines a relation between $C$ and $D$ in \eqref{HII}.
\item[step 3.] write a linear combination of the cavity modes corresponding to the contributions for the different values of $n$ and take into account the conditions on the surface $y=0$ (for $x\in[-d/2,d/2]$) using also \eqref{HI}.
The impedance condition yields each reflection coefficient $R_m$, $m\in\bbZ$ in terms of the amplitudes $A_n$, $n\in\bbN$, of each cavity mode.
The continuity of the fields above the cavity (for $x\in[-w/2,w/2]$) yields cavity mode amplitudes in terms of the reflection coefficients. 
It leads to a linear system relating the cavity mode amplitudes, which we can solve, and from which we can deduce the values of the reflection coefficients.  
\end{description}

\subsection{Physical parameters}
\label{sec.param}

In the numerical experiments, we use exactly the same parameters as in \cite{Barbara-Quemerais-Bustarret-LopezRios-Fournier03}, namely a cavity width of $w=0.75\micron$, a periodicity of $d=1.75\micron$ and a cavity height of $h=1.11\micron$.
The incident wave angle is $\theta=7.5^\circ$ and the wavenumber varies between $400$ and $7400\cminv$. To compute $k_0$ in S.I. units, we have to multiply these values by $200\pi$, and therefore $k_0$ ranges from $0.25~10^6$ to $4.6~10^6\radminv$. 
The monomode mixed case corresponds to sub-half-wavelength cavities: $k_0<\pi/w = 4.2~10^6\radminv = 6666\cminv$.
We refer to \cite{Palik85} for values of $\eps_m$ for gold in this frequency range.

\section{Determining cavity modes}
\label{sec.modes}

The profile of the cavity modes is given by the first two steps given in Section \ref{Sec.method}.
We first determine the $x$-dependance of the cavity modes \eqref{HII}, i.e. the values of $\mu$ such that the functions $u(x)=Ae^{ik_0\mu x}+Be^{-ik_0\mu x}$ are solution to 
\begin{equation*}
\left\{
\begin{aligned}
u''(x) + k_0^2\mu^2 u(x) & = 0, \\
u'(-w/2)+ik_0\xi u(-w/2) & = 0, \\
u'(w/2)-ik_0\xi u(-w/2) & = 0.
\end{aligned}
\right.
\end{equation*}

In the (P), (M$_0$), and (M) cases, the wall boundary condition is a perfect metal condition, therefore $\xi=0$, and this problem is self-adjoint, leading to equally spaced horizontal modes: $\mu_n=n\pi/k_0w$ for $n\in\bbN$, and the simple relation $B=(-1)^nA$. \\

In the (R$_0$) and (R) cases, the above eigenvalue problem case can be seen as a perturbation of the self-adjoint case $\xi=0$. 
We once more have $B=\sigma_n A$, where $\sigma_n = \pm1$, but the values of $\mu_n$ are not explicit any more but only solution to a transcendental equation
\begin{equation}
\label{eq.modes}
\exp(i\mu_nk_0w) = \sigma_n \frac{\mu_n+\xi}{\mu_n-\xi}.
\end{equation}

For all $n\in\bbN$, we define the cavity mode
\begin{equation}
\label{eq.cavity}
\psi_n(x) = \frac12(-i)^n \left(e^{i\mu_nk_0x}+\sigma_ne^{-i\mu_nk_0x}\right).
\end{equation}

\subsection{Numerical determination of modes}

In the (R$_0$) case, $\eps_m$ is real and $\xi$ is therefore purely imaginary. 
The corresponding values of $\mu_n$ are real and the equation can be solved in the real domain. 
This can be easily done with a good precision by using e.g. dichotomy and a very fine step.
These first numerical results show that $\mu_n$ is very close to $n\pi/k_0w$ for $n\in\bbN^*$. 

If now $\Im(\eps_m)$ is taken into account (as in the (R) case), $\txi\equiv k_0w\xi\equiv\zeta+i\eta$ and $\mu_n$ is no more real: the numerical determination in the complex domain is a bit more delicate. It involves some determinant which has to be positive to find a relevant solution and it happens to be positive only around $n\pi/k_0w$, $n\in\bbN$.

We therefore look for a solution of \eqref{eq.modes} in the form $\mu_n=(n\pi+\alpha_n)/k_0w$, where $\alpha_n$ is small.
It is then clear that $\sigma_n=(-1)^n$ and $\alpha_n$ is solution to the "simpler" equation
\begin{equation}
\label{eq2.modes}
(n\pi+\alpha_n)\tan\left(\frac{\alpha_n}2\right) = -i\txi.
\end{equation}

\subsection{Asymptotic determination of modes}

For $n\in\bbN^*$ and $\txi$ in some bounded set of the complex plane (which is valid for our application), we write an expansion of $\alpha_n$ in terms of $\txi/n$ and find that
\begin{equation*}
\begin{aligned}
\alpha_n = & -\frac{2i}\pi \frac{\txi}n + \frac4{n\pi^3} \frac{\txi^2}{n^2} + \left(\frac{16i}{n^2\pi^5} - \frac{2i}{3\pi^3}\right) \frac{\txi^3}{n^3} - \left(\frac{80}{n^3\pi^7} - \frac{16}{3n\pi^5}\right) \frac{\txi^4}{n^4} \\
& - \left(\frac{448i}{n^4\pi^9} - \frac{40i}{n^2\pi^7} + \frac{2i}{5\pi^5}\right) \frac{\txi^5}{n^5} + o\left(\frac{\txi^5}{n^5}\right).
\end{aligned}
\end{equation*}
We have to go relatively far in the expansion since $\txi$ is not very small as shows Table \ref{tab.xi}. 

\begin{table}[h]
\begin{center}
\begin{tabular}{r|ccccccc}
\hline
$k_0$ & 1000 & 2000 & 3000 & 4000 & 5000 & 6000 & 7000 \\
\hline
$\zeta$ & 0.0018 & 0.0038 & 0.0064 & 0.0094 & 0.0126 & 0.0159 & 0.0194 \\
\hline
$\eta$ & -0.0083 & -0.0297 & -0.0668 & -0.1191 & -0.1864 & -0.2663 & -0.3613 \\
\hline
\end{tabular}
\end{center}
\caption{\label{tab.xi}Value of $\Re(\txi)$ and $\Im(\txi)$ with respect to the wave number $k_0$ (in $\cminv$), derived from values of $\eps_m$ \cite{Palik85}.}
\end{table}

The dependance of modes $n=1$ and $n=2$ of the (R) case with respect to the wavenumber is shown on Figure \ref{fig.mu_n}. 

\begin{figure}[htbp]
\begin{center}
\begin{minipage}{.45\linewidth}
\includegraphics[width=\textwidth]{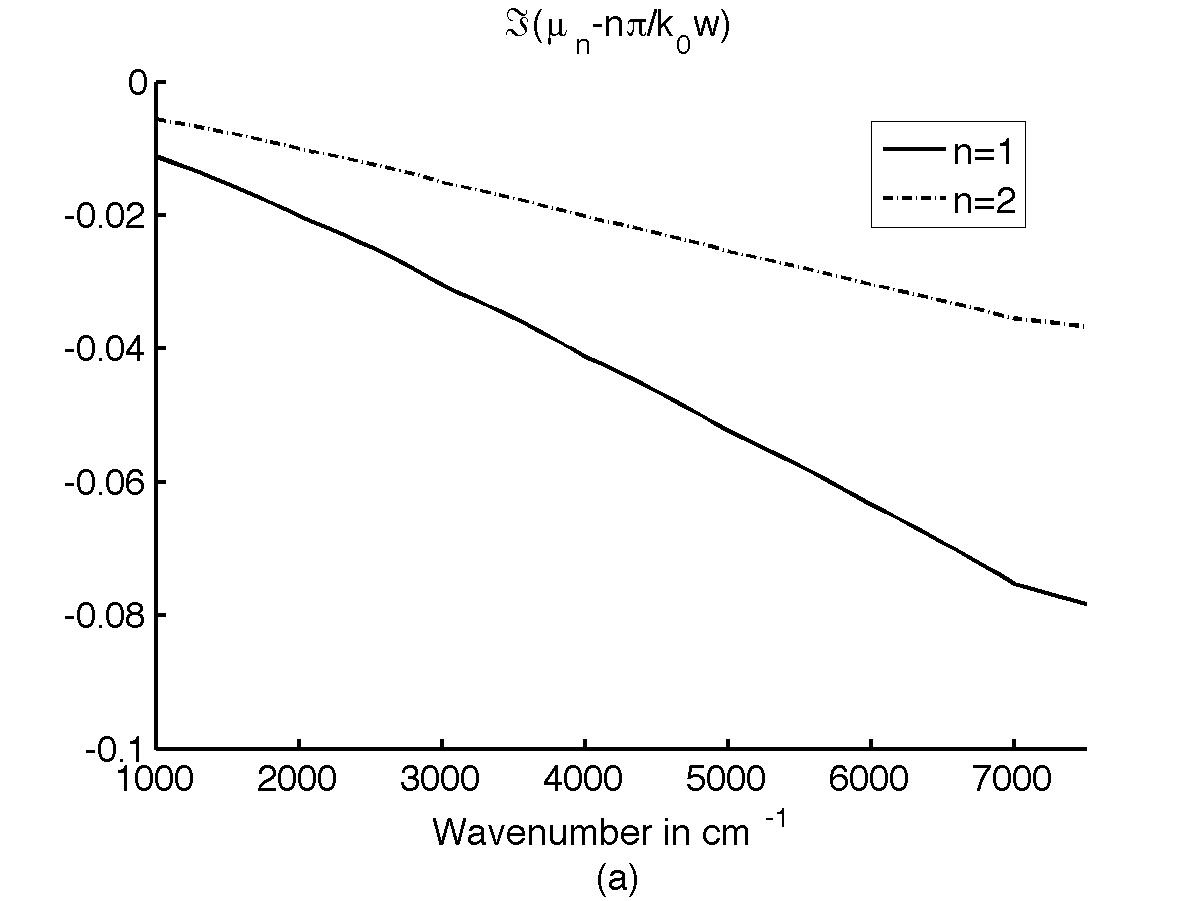}
\end{minipage}
\hfill
\begin{minipage}{.45\linewidth}
\includegraphics[width=\textwidth]{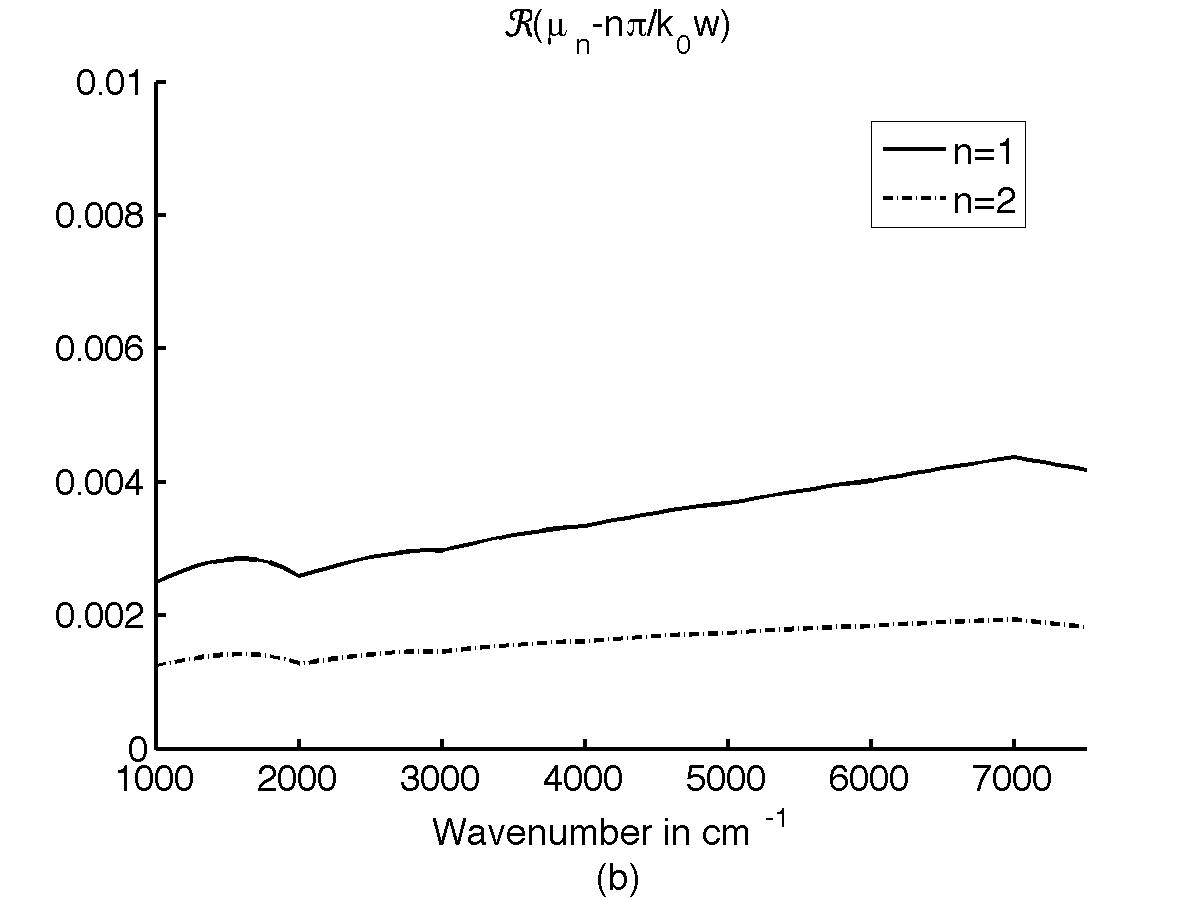}
\end{minipage}
\caption{\label{fig.mu_n} Values of $\Re(\mu_n-n\pi/k_0w)$ (a) and $\Im(\mu_n-n\pi/k_0w)$ (b) with respect to the wave number for modes $n=1$ and 2.}
\end{center}
\end{figure}

\subsection{Looking for a "fundamental" mode}

In the previous expansions we have not treated the case $n=0$.
We are guided by the simple computation
\begin{equation*}
\alpha_0\tan\left(\frac{\alpha_0}2\right) = -i\txi
\hspace{1cm}\Rightarrow\hspace{1cm} i\alpha_0^2 \simeq 2\txi.
\end{equation*}
This is clearly impossible in the (R$_0$) case, where this is equivalent to $\alpha_0^2 \simeq 2\eta$, where $\eta<0$ (see Table \ref{tab.xi}) and $\alpha_0$ should be real. 
On the contrary, it is possible to solve this in the complex domain in the (R) case, it is natural to seek an expansion of the first mode $\alpha_0$ in terms of $\chi=\sqrt{\txi}$, yielding only odd terms:
\begin{equation*}
\alpha_0 = (1-i) \left(\chi +  \frac{i}{12} \chi^3 - \frac{11}{1440} \chi^5 - \frac{17i}{40320} \chi^7 - \frac{281}{9676800} \chi^9\right)+ o(\chi^{10}).
\end{equation*}

The dependance of this first mode with respect to the wavenumber in the (R) case is shown on Figure \ref{fig.mu_0}. 

\begin{figure}[htbp]
\begin{center}
\begin{minipage}{.45\linewidth}
\includegraphics[width=\textwidth]{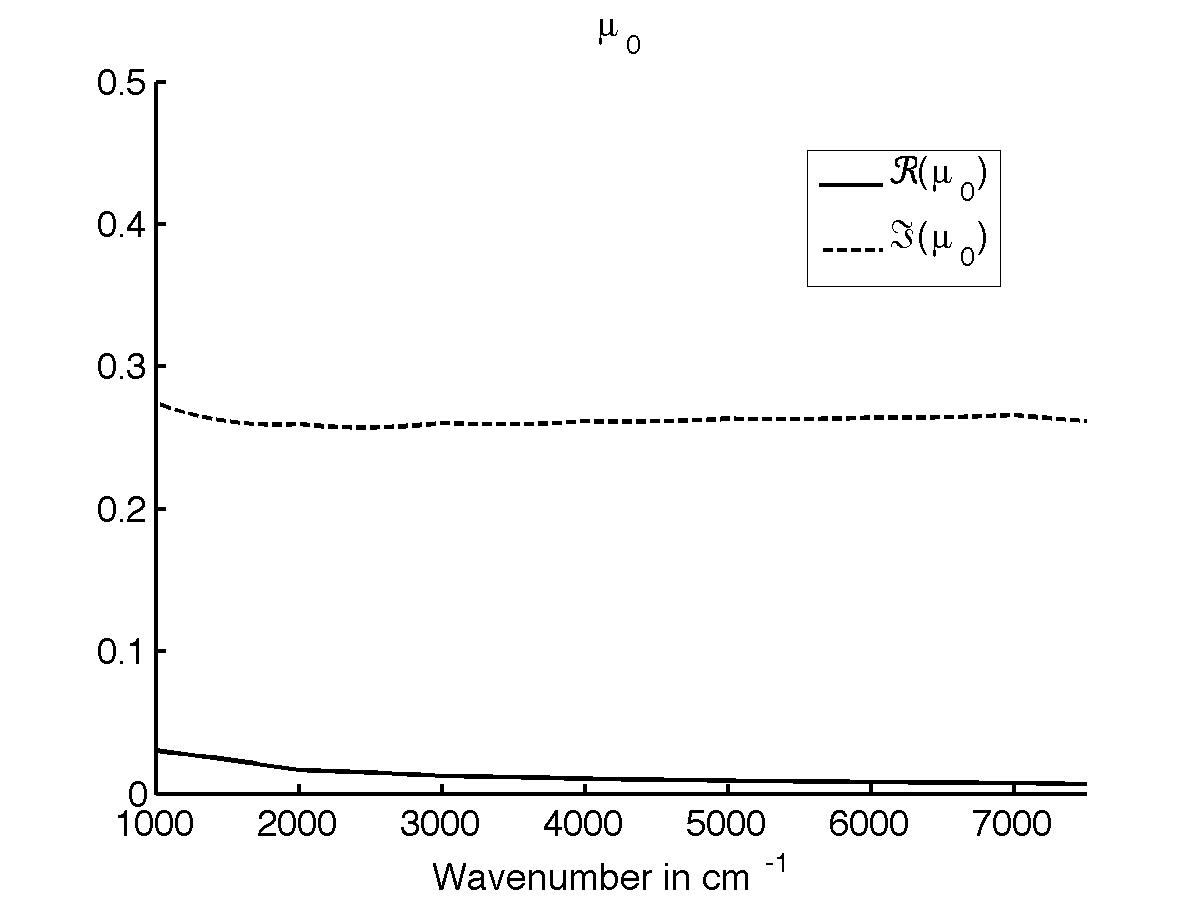}
\end{minipage}
\caption{\label{fig.mu_0} Values of $\Re(\mu_0)$ and $\Im(\mu_0)$ with respect to the wave number.}
\end{center}
\end{figure}

The fundamental mode plotted on Figure \ref{fig.mu_0} shows an almost constant imaginary part and a larger real part for low wavenumbers. 
This behavior is quite different from the higher order modes (the first two being displayed on Figure \ref{fig.mu_n}) where $|\alpha_n|$ increases with $|\txi|$. 
If we had plotted $\mu_n-n\pi/k_0w$ in the (R$_0$) case, the picture would have been very close to the real parts displayed on Figure \ref{fig.mu_n}(a).
This behavior of the fundamental mode clearly does not "pass to the limit" as the permittivity becomes real, and this is another evidence of the absence of a small mode in the (R$_0$) case.
The (R$_0$) case, which could be considered as more precise as the (M$_0$) is in fact not at all suitable for modal computations (see also the first numerical test cases in Section \ref{sec.monomode}).

\subsection{Completeness of the cavity modes}

We now want to show that we can use the above defined modes as a basis of functions in the cavity.
Completeness of the eigenfunctions follows from the same argument as in the similar context of \cite{Hazard-Lenoir93}. 
We indeed follow the result of \cite{Kato95}:
\begin{theorem}[\cite{Kato95}, theorem V.2.20]
\label{Th.Kato}
Let $\{\phi_j\}$ be a complete orthonormal family in a separable Hilbert space $H$ and let $\{\psi_j\}$ be a sequence such that $r^2=\sum_{j=1}^\infty\|\psi_j-\phi_j\|^2<\infty$. Then $\{\psi_j\}$ is a basis of $H$ if $\sum_{j=1}^\infty \eta_j\psi_j=0$ implies that $\eta_j=0$ for all $j$.
\end{theorem}

\begin{proposition}
\label{prop.basis}
The cavity modes $\{\psi_n,\ n\in\bbN\}$ form a basis of $L^2(-w/2,w/2)$.
\end{proposition}

\begin{proof}
We denote $\check\psi_n(x)=\cos(n\pi(x+w/2)/w)$, $n\in\bbN$, the functions $\psi_n$ defined by Equation \eqref{eq.cavity} in the cases (P), (M$_0$) or (M) where $\mu_n=n\pi/k_0w$.
To prove Proposition \ref{prop.basis} we renormalize our functions $\psi_n$ and $\check\psi_n$, to $\rho_n\psi_n$ and $\check\rho_n\check\psi_n$. 
Taking $\check\rho_0=1$ and $\check\rho_n=\sqrt2$ for $n>0$, it is a classical result that $\check\rho_n\check\psi_n$ is an orthonormal basis of $L^2(-w/2,w/2)$.
For each $n>0$, we want to find $\rho_n$ such that $d_n=\|\rho_n\psi_n-\check\rho_n\check\psi_n\|$ is minimal.
We first compute 
\begin{equation*}
d_n^2 = |\rho_n|^2 I_n - 2\check\rho_n \Re(\rho_nJ_n) + 1,
\end{equation*}
where
\begin{equation*}
I_n = \frac1w \int_{-w/2}^{w/2} \psi_n(x)\bar\psi_n(x) dx
\textrm{ and }
J_n = \frac1w \int_{-w/2}^{w/2} \psi_n(x)\check\psi_n(x) dx.
\end{equation*}
The minimum value of $d_n^2$ , clearly obtained when $\rho_n=\check\rho_n\bar J_n/I_n$, is 
\begin{equation*}
d_{n,\rm min}^2 = 1 - 2 \frac{|J_n|^2}{I_n}.
\end{equation*}
An explicit computation leads to
\begin{eqnarray*}
I_n & = & \frac12\left(\frac{e^{\Im(\alpha_n)}-e^{-\Im(\alpha_n)}}{2\Im(\alpha_n)}+\frac{e^{i\Re(\alpha_n)}-e^{-i\Re(\alpha_n)}}{2i(n\pi+\Re(\alpha_n))}\right), \\
J_n & = & \frac{(n\pi+\alpha_n)\sin(\alpha_n/2)}{\alpha_n(n\pi+\alpha_n/2)}.
\end{eqnarray*}
Using the previously obtained expansion of $\alpha_n$
\begin{equation*}
I_n = \frac12 + \frac{\zeta^2}{3\pi^2n^2} + \frac{\eta}{\pi^2n^2} + o\left(\frac1{n^2}\right)
\textrm{ and }
J_n = \frac12 + \frac{\eta}{2\pi^2n^2} - i\frac{\zeta}{2\pi^2n^2} + o\left(\frac1{n^2}\right),
\end{equation*}
and finally
\begin{equation*}
d_{n,\rm min}^2 = \frac{2\zeta^2}{3\pi^2n^2} + o\left(\frac1{n^2}\right).
\end{equation*}
We therefore have $\sum_n d_{n,\rm min}^2 < \infty$. \\

Now if we compute explicitely
\begin{equation*}
T_{nn'} = \frac1w \int_{-w/2}^{w/2} \psi_n(x)\psi_{n'}(x) dx,
\end{equation*}
the fact that $\mu_n$ and $\mu_{n'}$ both are solution to \eqref{eq2.modes} leads to $T_{nn'}=0$ if $n\neq n'$.
We simply denote $T_n\equiv T_{nn}$ in the sequel.
Therefore $\sum_{n=0}^\infty \eta_n\psi_n=0$ leads to $\eta_n=0$ for all $n$ and, following Theorem \ref{Th.Kato}, $\{\rho_n\psi_n\}$ (and of course also $\{\psi_n\}$) is a basis of $L^2(-w/2,w/2)$.
\end{proof}

\subsection{Cavity mode profile}
\label{sec.profile}

The computations of the second step are the same in all cases, and lead for each $n$ to the $n$th cavity mode profile:
\begin{equation*}
\calC_n (x,y) = \frac12 (-i)^n(e^{i\mu_nk_0x}+\sigma_n e^{-i\mu_nk_0x})
e^{-i\nu_nk_0h} (e^{i\nu_nk_0(y+h)}+ r_n e^{-i\nu_nk_0(y+h)}), 
\end{equation*}
where $r_n = (\nu_n+\xi)/(\nu_n-\xi)$.
The (P), (M$_0$), and (M) cases can also be cast like this with the specific values $\mu_n=n\pi/k_0w$.
The field in the cavity consists of a superposition of these modes:
\begin{equation*}
H^{II}_z(x,y) = \sum_{n\in\bbN} A_n \calC_n (x,y).
\end{equation*}
The determination of the square root $\nu_n=\sqrt{1-\mu_n^2}$ is of great importance.
Indeed, if we choose the usual determination, where $\Im(\nu_n)\geq0$, then $\calC(x,y)$ is huge if evanescent modes are taken into account.
This is shown on Figure \ref{fig.energy_pos}, where the logarithm of the energies $\calE_n$ of the first cavity modes is given: 
\begin{equation*}
\calE_n = \int_{x=-w/2}^{w/2} \int_{y=-h}^0 |\calC_n(x,y)|^2 dx dy. 
\end{equation*}

\begin{figure}[htbp]
\begin{center}
\begin{minipage}{.45\linewidth}
\includegraphics[width=\textwidth]{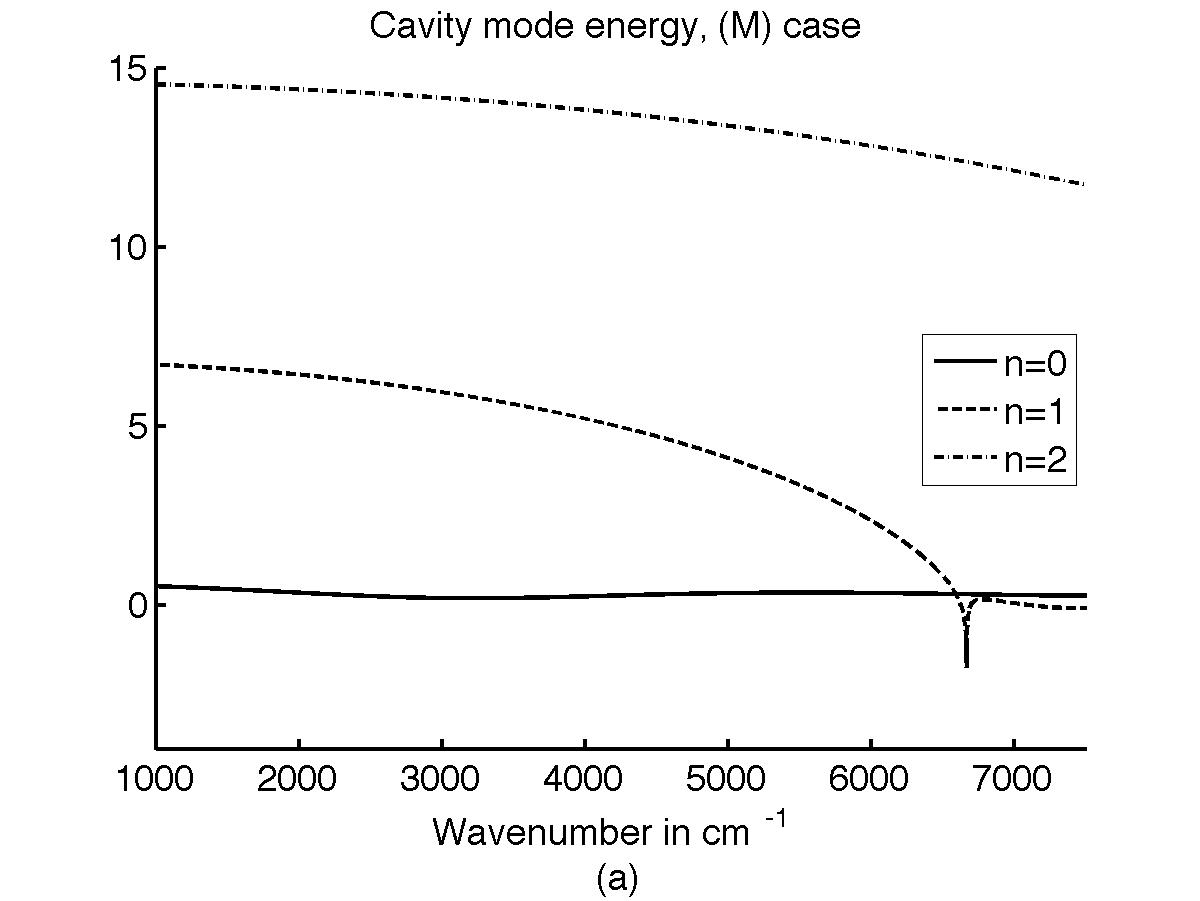}
\end{minipage}
\hfill
\begin{minipage}{.45\linewidth}
\includegraphics[width=\textwidth]{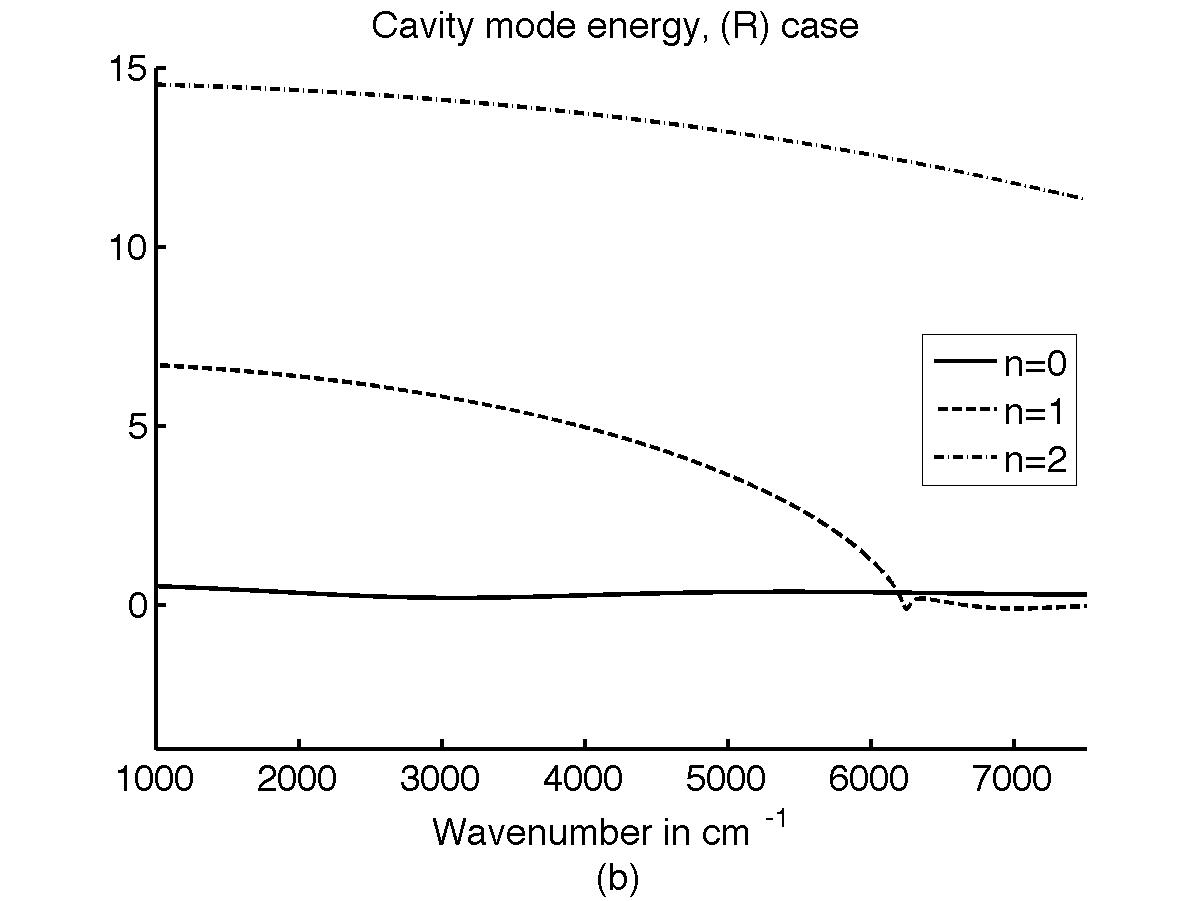}
\end{minipage}
\caption{\label{fig.energy_pos} $\log_{10}(\calE_n)$ for $\Im(\nu_n)\geq0$ and for $n=0$, 1, 2 in the (M) case (a), and the (R) case (b).}
\end{center}
\end{figure}

The role of the presence of the evanescent modes is clear: in the (M) case, above $k_0=6666\cminv$, the theoretical limit of the monomode case (as computed in Section~\ref{sec.param}), 
where only the first mode is non-evanescent, $\calC_1$ has a moderate size  (Figure \ref{fig.energy_pos}(a)).
In the (R) case, the notion of an evanescent mode is slightly different since $\mu_n$ has always a real and an imaginary part.
We see on Figure \ref{fig.energy_pos}(b) that the upper limit of the (R)-monomode case is $k_0\simeq6250\cminv$.

Having such different magnitudes (contrast of about 15 decades) in the description of the total wave $H_z^{II}$ will lead to a very poor conditioning of the linear system we shall eventually invert (see Section \ref{sec.cond}). 

If we make the opposite choice and have $\Im(\nu_n)\leq0$, the magnitude of $\calE_n$ is always relatively moderate, as shown on Figure \ref{fig.energy_neg}.

\begin{figure}[htbp]
\begin{center}
\begin{minipage}{.45\linewidth}
\includegraphics[width=\textwidth]{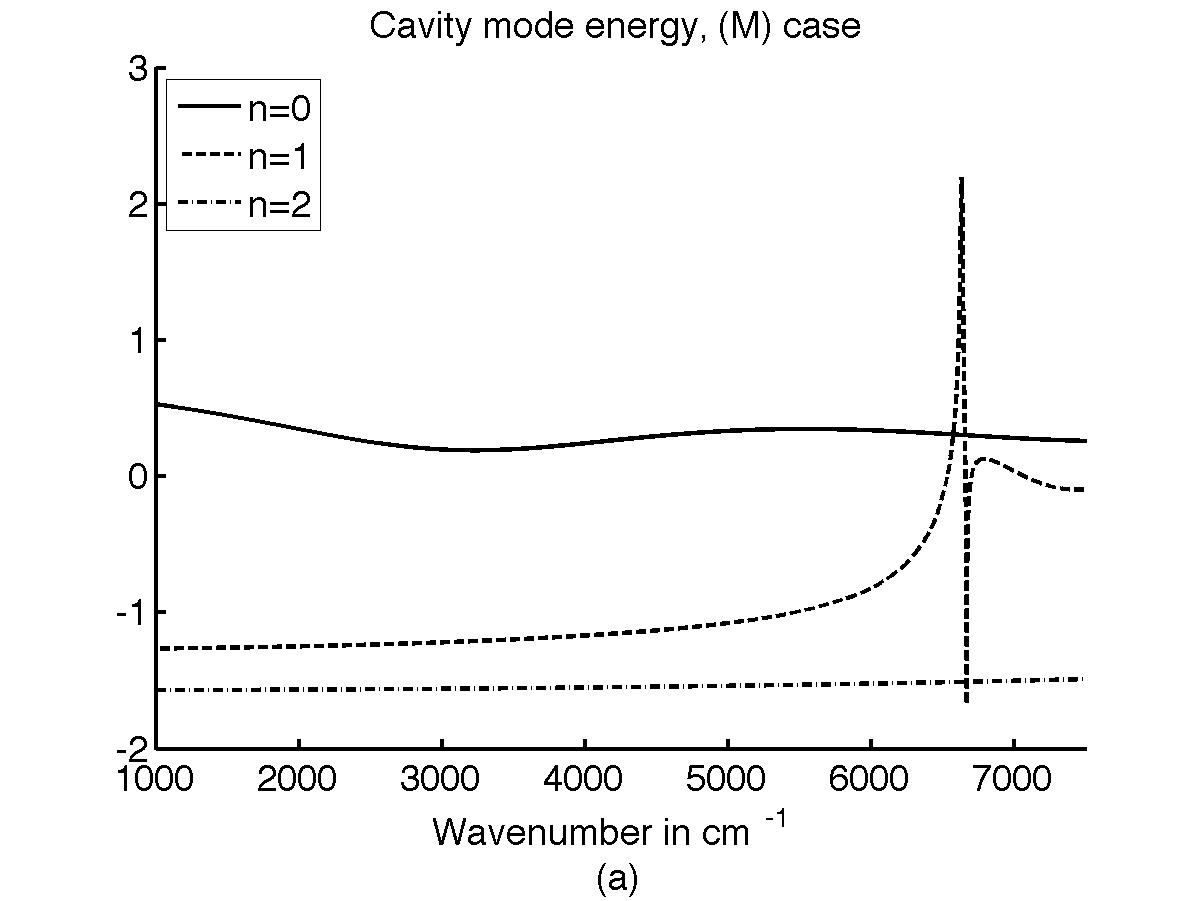}
\end{minipage}
\hfill
\begin{minipage}{.45\linewidth}
\includegraphics[width=\textwidth]{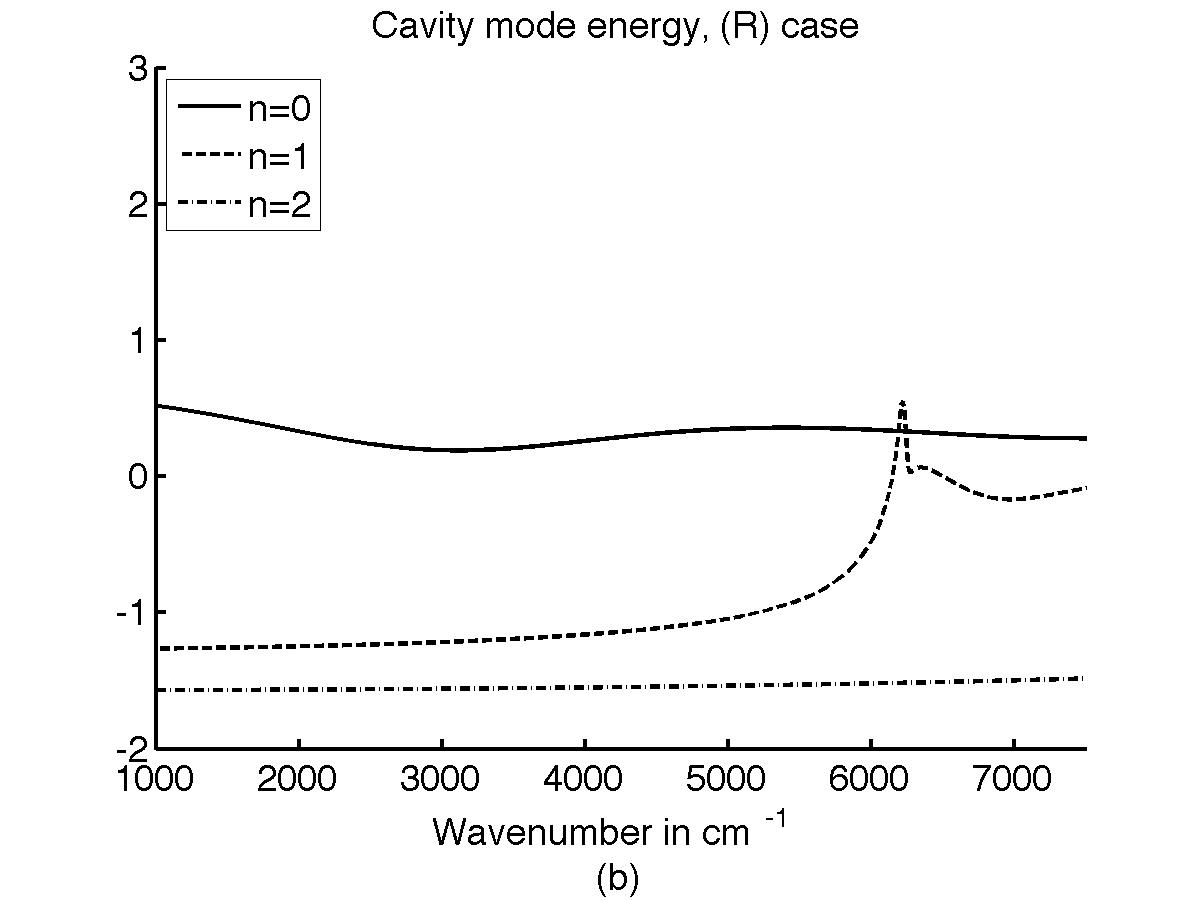}
\end{minipage}
\caption{\label{fig.energy_neg} $\log_{10}(\calE_n)$ for $\Im(\nu_n)\leq0$ and for $n=0$, 1, 2 in the (M) case (a), and the (R) case (b).}
\end{center}
\end{figure}

Now the difference of magnitude of the different modes at a given frequency is of the order of 100, which is quite acceptable.
Choosing the sign of $\Im(\nu_n)$ is only choosing a way to normalize the cavity modes.

\subsection{Towards an infinite dimensional system}

The impedance condition (the quantity $(\partial_y+ik_0\xi)H^{I}_z(x,0)$ is vanishing on $(-d/2,-w/2)\cup(w/2,d/2)$) and the continuity of the fields are projected in the $L^2_x$ sense in two ways: 
\begin{itemize}
\item firstly in $L^2(-d/2,d/2)$ on functions $\phi_m(x)=e^{ik_0\gamma_mx}$, filtering this oscillation in $H^I_z(x,0)$ to recover the reflection coefficient $R_m$; 
\item secondly in $L^2(-w/2,w/2)$ on functions $\psi_n(x)=(-i)^n(e^{i\mu_nk_0x} + \sigma_n e^{-i\mu_nk_0x})/2$, yielding the cavity mode coefficients $A_n$.
\end{itemize}
We therefore compute
\begin{equation*}
\begin{aligned}
\frac1d \int_{-d/2}^{d/2} 
& \left(\partial_y H^{I}_z(x,0^+) +ik_0 \xi H^{I}_z(x,0^+)\right) \bar\phi_m(x)\ dx \\
&= \frac1d \int_{-w/2}^{w/2} 
\left(\partial_y H^{II}_z(x,0^-)+ik_0 \xi H^{II}_z(x,0^-)\right)
\bar\phi_m(x)\ dx 
\end{aligned}
\end{equation*}
which yields the reflection mode amplitudes in terms of the cavity mode amplitudes:
\begin{equation}
\label{Rm}
R_m
= \frac{\beta_0-\xi}{\beta_0+\xi} \delta_{m0} 
+ \frac{\Gamma}{\beta_m+\xi} 
\sum_{n\in\bbN} A_n \sigma_n S_{mn} (\nu_n+\xi) [1 - e^{-2i\nu_nk_0h}],
\end{equation}
where $\Gamma=w/d$ is the aspect ratio, $\delta_{m0}$ is the Kronecker delta, and 
\begin{equation*}
S_{mn} = \frac1w \int_{-w/2}^{w/2} \psi_n(x)\bar\phi_m(x)\ dx.
\end{equation*}
We also project the continuity condition above the cavity $H^I(x,0)=H^{II}(x,0)$ for $-w/2\leq x\leq w/2$ onto the functions $\psi_n(x)$
\begin{equation*}
\frac1w \int_{-w/2}^{w/2} H^{I}_z(x,0^+) \psi_n(x)\ dx
= \frac1w \int_{-w/2}^{w/2} H^{II}_z(x,0^-) \psi_n(x)\ dx
\end{equation*}
which yields the cavity mode amplitudes in terms of the reflection mode amplitudes:
\begin{equation}
\label{An}
A_n = \frac1{T_n(1+r_n e^{-2i\nu_nk_0h})} \sum_{m\in\bbZ} S_{mn}(\delta_{m0} + R_{m}),
\end{equation}
where
\begin{equation*}
T_n = \frac1w \int_{-w/2}^{w/2} \psi_n^2(x) dx = \frac12\left(\sigma_n\sinc(\mu_nk_0w)+1\right),
\end{equation*}
with as usual $\sinc(z) = \sin(z)/z$, for all $z\in\bbC$.
We can plug \eqref{Rm} in \eqref{An} and find an infinite dimensional linear system $\calM A = \calS$.

\section{Numerical experiments and mode amplitudes}
\label{sec.num}

To solve the linear system numerically an assumption has to be made on the number of relevant modes, which leads to a finite dimensional system.
Restricting the number of modes destroys the continuity at the interface between the two regions (I) and (II).
Let $N$ be the number of cavity modes and $M$ the number of reflected modes kept for the computations.
One possible way to choose these values is to keep only propagative waves, i.e. reflected waves for which $\beta_m$ is real (and cavity waves for which $\nu_n$ is real in the mixed case). 

\subsection{Conditioning}
\label{sec.cond}

The conditioning of matrix $\calM$ is an important issue, except in the monomode case ($N=1$) where $\calM$ is a scalar.
We have foreseen in Section \ref{sec.profile} that taking $\Im(\nu_n)\leq0$ certainly leads to better conditioned systems. 
This is indeed the case as shown on Figure \ref{fig.condition} where we show the conditioning of matrix $\calM$ for the (M) case (the (R) case would be similar but slightly less critical) for both $\Im(\nu_n)\geq0$ and $\Im(\nu_n)\leq0$ with $M=5$ and $N=1$, 2 and 3.

\begin{figure}[htbp]
\begin{center}
\begin{minipage}{.45\linewidth}
\includegraphics[width=\textwidth]{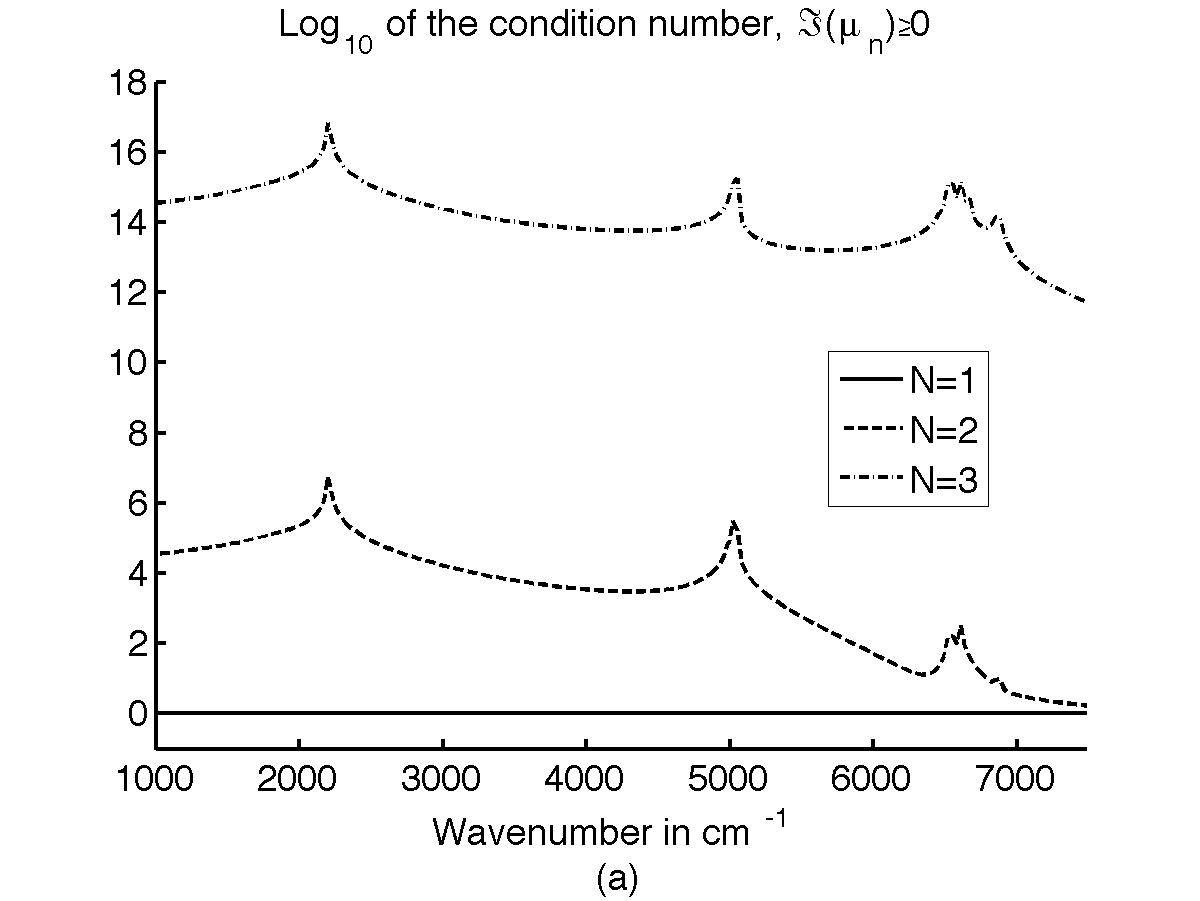}
\end{minipage}
\hfill
\begin{minipage}{.45\linewidth}
\includegraphics[width=\textwidth]{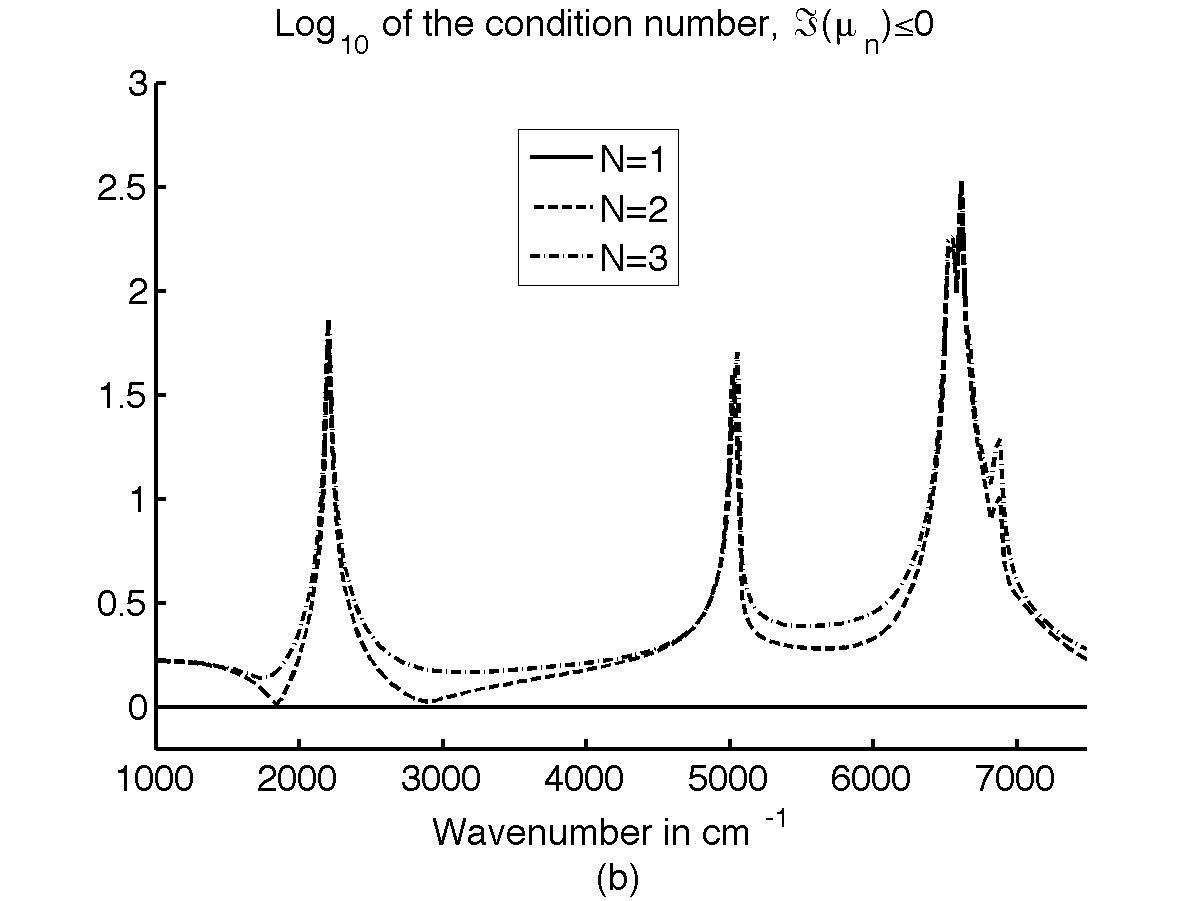}
\end{minipage}
\caption{\label{fig.condition}Impact of the sign of $\Im(\nu_n)$ and the number of cavity modes kept in the computations on the condition number in the (M) case: (a) $\Im(\nu_n)\geq0$; (b) $\Im(\nu_n)\leq0$.}
\end{center}
\end{figure}

We see that the conditioning of $\calM$ is not sensitive to $N$ if $\Im(\nu_n)\leq0$.
We do not show the plots here, but the conditioning is also not sensitive to $M$. 
A plot with $M=41$ would be very similar. 

We stress the fact that the solutions are not affected by the very poor conditioning because the entries of $\calS$ also show magnitude differences if $\Im(\nu_n)\geq0$.
It seems however better to use $\Im(\nu_n)\leq0$ and a well-conditioned system.

Another way to have a reasonable conditioning is to take advantage of the fact that reflected modes and cavity modes play a symmetric role. 
If we replace the system in the variable $A$ ($\calM A = \calS$) by a system $\tilde\calM R = \tilde\calS$, the conditioning is very similar to that displayed on Figure \ref{fig.condition}(b).

Let $U=\{m\in\bbZ,\ \beta_m\in\bbR\}$.
With our physical values $U=\{0\}$ up to $k_0=5054\cminv$, then $U=\{-1,0\}$ up to $k_0=6572\cminv$ and $U=\{-1,0,1\}$ above.
Peaks on Figure \ref{fig.condition} seem to occur at these values.

\subsection{Validation by reflectance experiments}

To validate our approach, we try to reproduce experiments with the physical parameters given in Section \ref{sec.param}. 
It is possible to measure the propagative reflected waves and therefore we can define the specular reflectance $|R_0|^2$ and the total reflectance $\sum_U |R_m|^2\beta_m/\beta_0$. 

\subsubsection{Monomode experiments}
\label{sec.monomode}

For narrow cavities ($w<\pi/k_0$) it is admitted that only a single mode can be kept in the cavity. 
We then have
\begin{equation}
\label{eq.A0}
A_0 \left(T_0(1+r_0e^{-2i\nu_0k_0h})-\Gamma(1+\xi)(1-e^{-2i\nu_0k_0h}) \sum_{m\in\bbZ} \frac{s_m^2}{\beta_m+\xi} \right) = \frac{2\beta_0s_0}{\beta_0+\xi},
\end{equation}
where 
\begin{eqnarray*}
s_m & \equiv & S_{m0} 
= \frac12 \sinc((\mu_0+\gamma_m) k_0w/2) + \frac12 \sinc((\mu_0-\gamma_m) k_0w/2), \\
T_0 & = & \frac12 (\sinc(\mu_0k_0w) + 1).
\end{eqnarray*}
Equation \eqref{eq.A0} is simpler in the mixed case where
$\mu_0=0$ and $\nu_0=1$, and therefore $s_m=\sinc(\gamma_mk_0w/2)$ and $T_0=1$.

Figure \ref{fig.monomode} reproduces figure 3 in \cite{Barbara-Quemerais-Bustarret-LopezRios-Fournier03}. 
In this reference the (M) monomode case is used and compared to experimental results. 
On figure \ref{fig.monomode}(a) we add the curves concerning the (M$_0$) and (P) cases.
On figure \ref{fig.monomode}(b) we compare with our (R$_0$) and (R) computations.

\begin{figure}[htbp]
\begin{center}
\begin{minipage}{.45\linewidth}
\includegraphics[width=\textwidth]{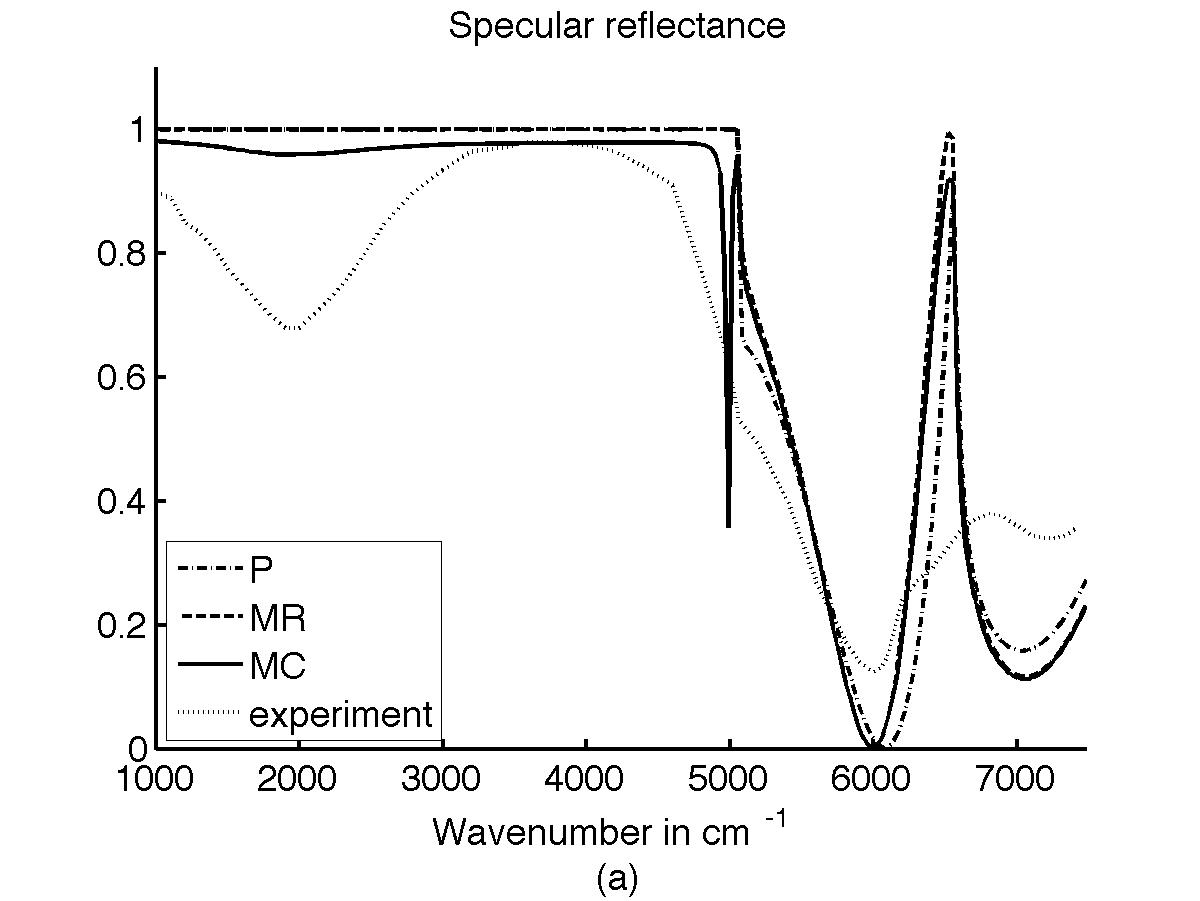}
\end{minipage}
\hfill
\begin{minipage}{.45\linewidth}
\includegraphics[width=\textwidth]{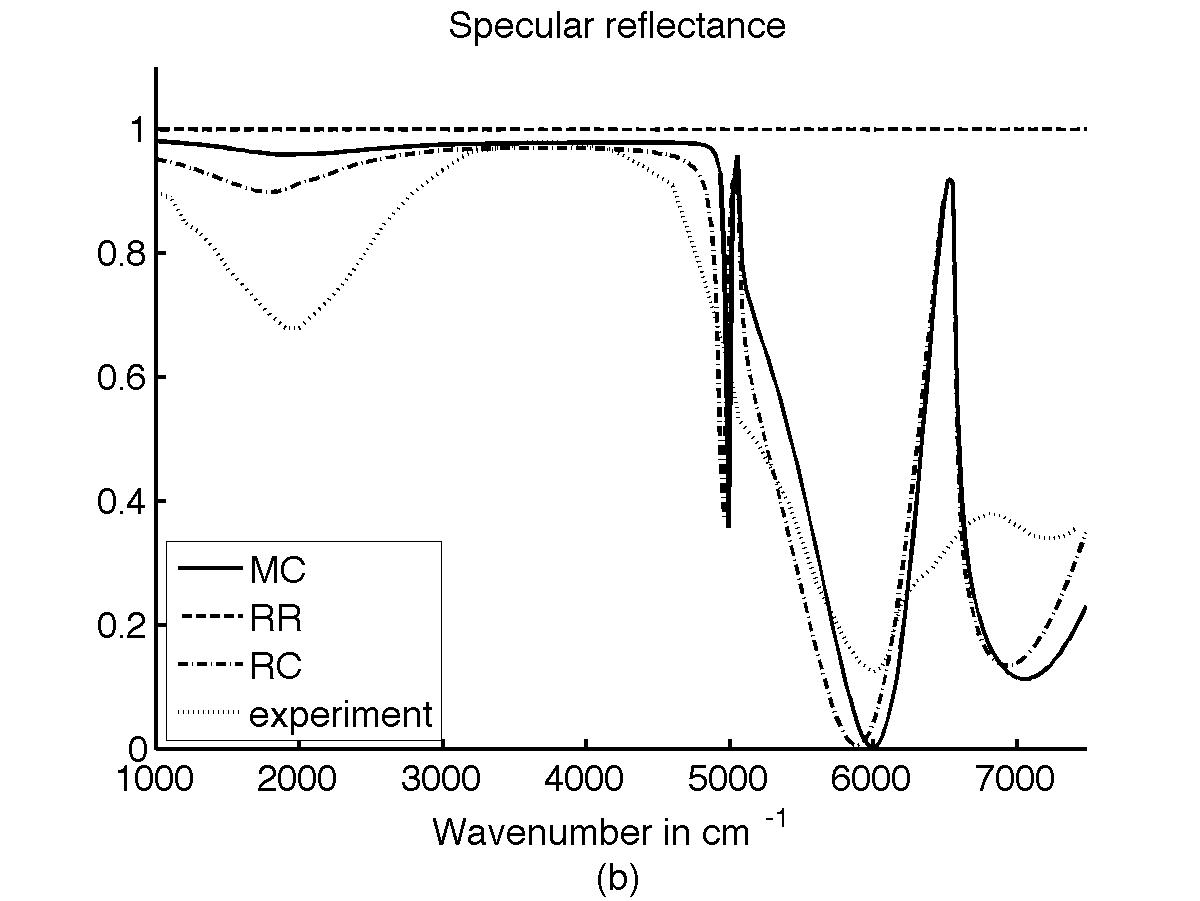}
\end{minipage}
\caption{\label{fig.monomode}(a) Specular reflectance for the (P), (M$_0$) and (M) cases, comparison with experimental results \cite{Barbara-Quemerais-Bustarret-LopezRios-Fournier03}; (b) Specular reflectance for the real (R$_0$) and (R) cases compared to the (M) case and experimental results.}
\end{center}
\end{figure}

We see on these plots that the perfect and mixed cases behave relatively similarly, at least for large wavenumbers. 
The (P) and (M$_0$) cases totally miss the first dip close to $2000\cminv$.
As predicted in \cite{Barbara-Quemerais-Bustarret-LopezRios-Fournier03}, considering real metallic surfaces on the walls ((R) case) allows to better catch this feature.
The (R) computations also give a better location of the decreasing part around $5000\cminv$.
The lack of fundamental mode totally disqualifies the (M$_0$) case for which the incident wave is totally reflected in the specular direction, whatever the wavenumber. 

For the next tests we will now only compare the (M) and (R) cases, i.e. only compare the impact of the choice of the boundary condition for a real metal (with a complex permittivity).

\subsubsection{Influence of the number of reflected waves}

The computations displayed on Figure \ref{fig.monomode} have been performed with one cavity mode ($n=0$) and 41 reflected waves $m=-20,\dots,20$. 
We may ask about the incidence of this choice.

First let us investigate the case of reflected waves. 
We remain in the monomo\-de case but make different choices for the number of reflected waves $M$.
Figure \ref{fig.M} shows the specular reflectance in the (M) and (R) cases with $M=5$ i.e. $m=-2,\dots,2$ (which yield very close results to the $M=41$ case used on Figure \ref{fig.monomode}), $M=1$ i.e. $m=0$ (only the specular reflected wave) and $m\in U$.

\begin{figure}[htbp]
\begin{center}
\begin{minipage}{.45\linewidth}
\includegraphics[width=\textwidth]{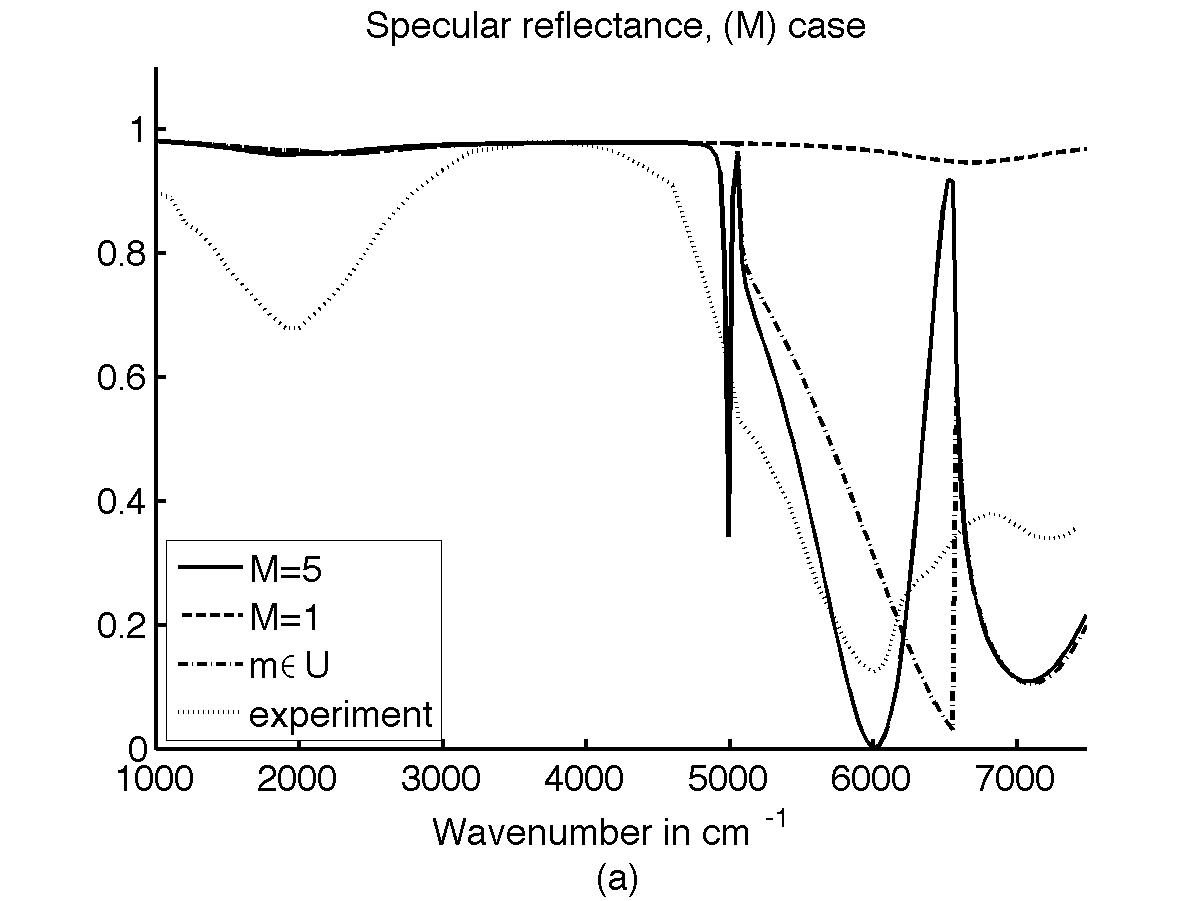}
\end{minipage}
\hfill
\begin{minipage}{.45\linewidth}
\includegraphics[width=\textwidth]{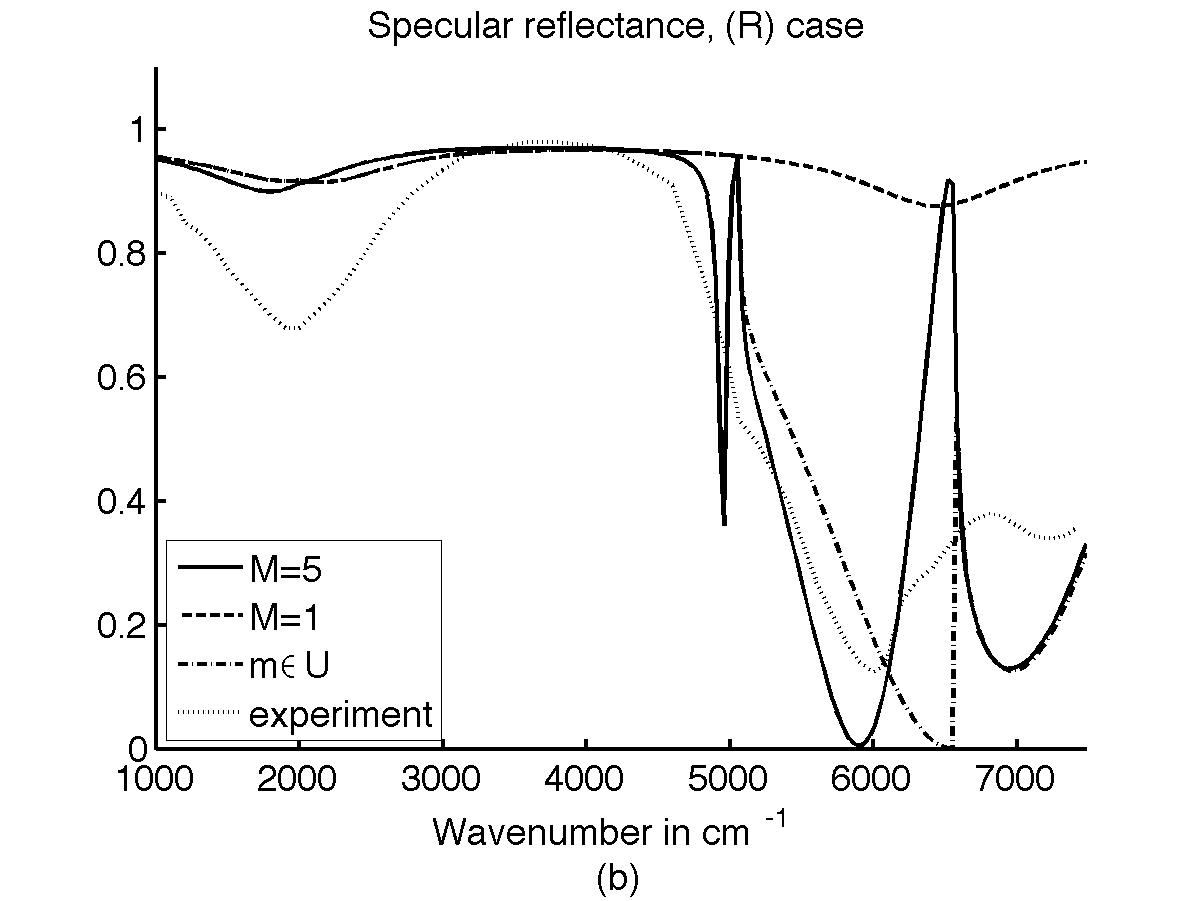}
\end{minipage}
\caption{\label{fig.M}Impact of the number of reflected waves kept in the computations on the specular reflectance in the (M) case (a), and the (R) case (b). For each case, we used 5 reflected waves ($m=-2,\dots,2$), the specular reflected wave ($m=0$) or all the propagative reflected waves ($m\in U$).}
\end{center}
\end{figure}

The comparison of the (M) and (R) cases leads to the same conclusion as in the mono\-mode test case.
They both behave in the same way regarding the impact of the choice of the number of reflected wave. 
Keeping only the exact number of propagative waves does not seem enough to describe completely the phenomena, let alone keeping only the specular wave. 
This huge impact can be clearly seen by considering the discontinuity at $k_0=6572\cminv$ when using $m\in U$, i.e. when changing the number of reflected waves at this wavenumber.

\subsubsection{Influence of the number of cavity modes}

Now we investigate the impact of the number of cavity modes.
Figure \ref{fig.N} shows the specular reflectance in the (M) and (R) cases with $N=1$, 2 and 3 modes.
Recall that the monomode assumption is \textit{a priori} valid up to $k_0=6666\cminv$.
Considering the previous test case we have taken $M=5$ here.

\begin{figure}[htbp]
\begin{center}
\begin{minipage}{.45\linewidth}
\includegraphics[width=\textwidth]{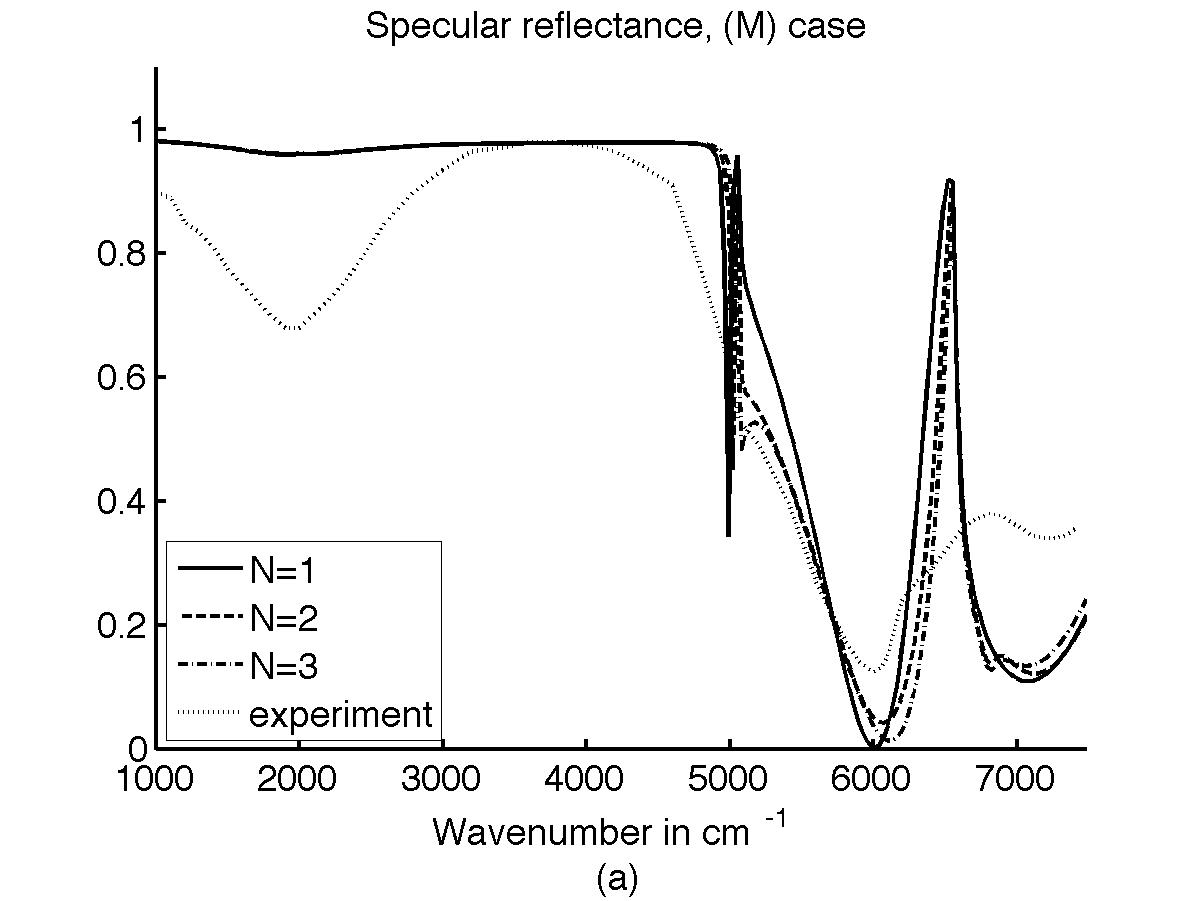}
\end{minipage}
\hfill
\begin{minipage}{.45\linewidth}
\includegraphics[width=\textwidth]{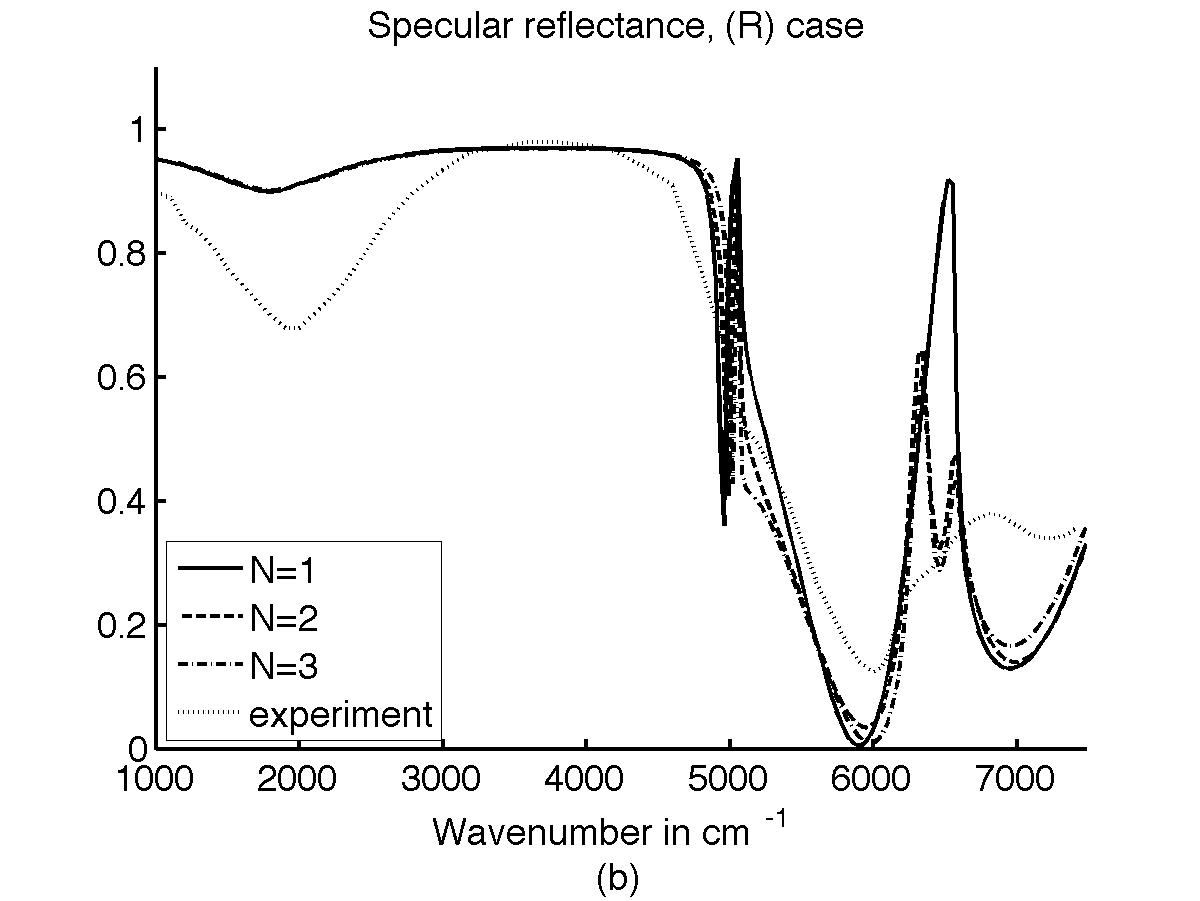}
\end{minipage}
\caption{\label{fig.N}Impact of the number of cavity modes kept in the computations on the specular reflectance in the (M) case (a), and the (R) case (b). For each case, we used 1 mode ($n=0)$, 2 modes ($n=0,1$) and 3 modes ($n=0,1,2$).}
\end{center}
\end{figure}

In the (M) case, apart from already mentioned features, we notice that the approximation is much better by taking $N=2$ or 3 in the $5000$--$6000\cminv$ range.
The rest of the results are not so much affected by the number of cavity modes kept in the computation.

In the (R) case, we can note a dissociation between the dip below $5000\cminv$ and the decrease between $5000$ and $6000\cminv$.
When the monomode assumption is not valid anymore, the (R) case outperforms the (M) case if relevant cavity modes are kept.

\subsection{Amplitude of cavity modes}

Contrarily to reflected modes, it is impossible to measure the cavity modes without changing them.
We therefore lack experimental data to compare the numerical results with.
We can only hope that computations which yield fair results regarding reflected waves will also yield a correct estimation of the amplitude of cavity modes.
We only investigate modes $n=0$, 1 and 2.
Comparing the values of $A_n$ for the different methods would be pointless because it is multiplied by functions which can differ greatly in magnitude.
To be able to show the dependance with respect to the wavenumber we plot the renormalized cavity mode amplitude $a_n = A_n\sqrt{\calE_n}$.
 
Figure \ref{fig.A_0} shows $a_0$ in the (M) and the (R) cases for different values of $N$.

\begin{figure}[htbp]
\begin{center}
\begin{minipage}{.45\linewidth}
\includegraphics[width=\textwidth]{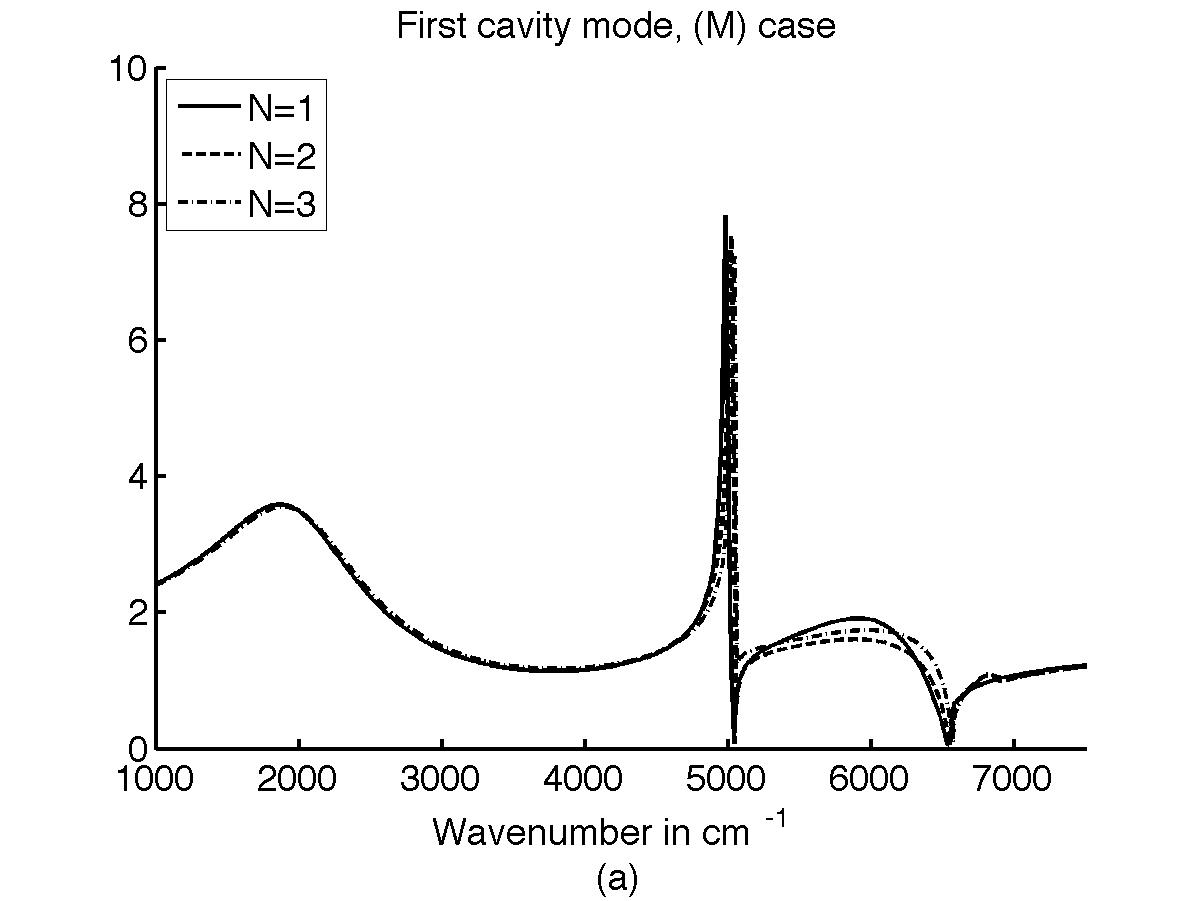}
\end{minipage}
\hfill
\begin{minipage}{.45\linewidth}
\includegraphics[width=\textwidth]{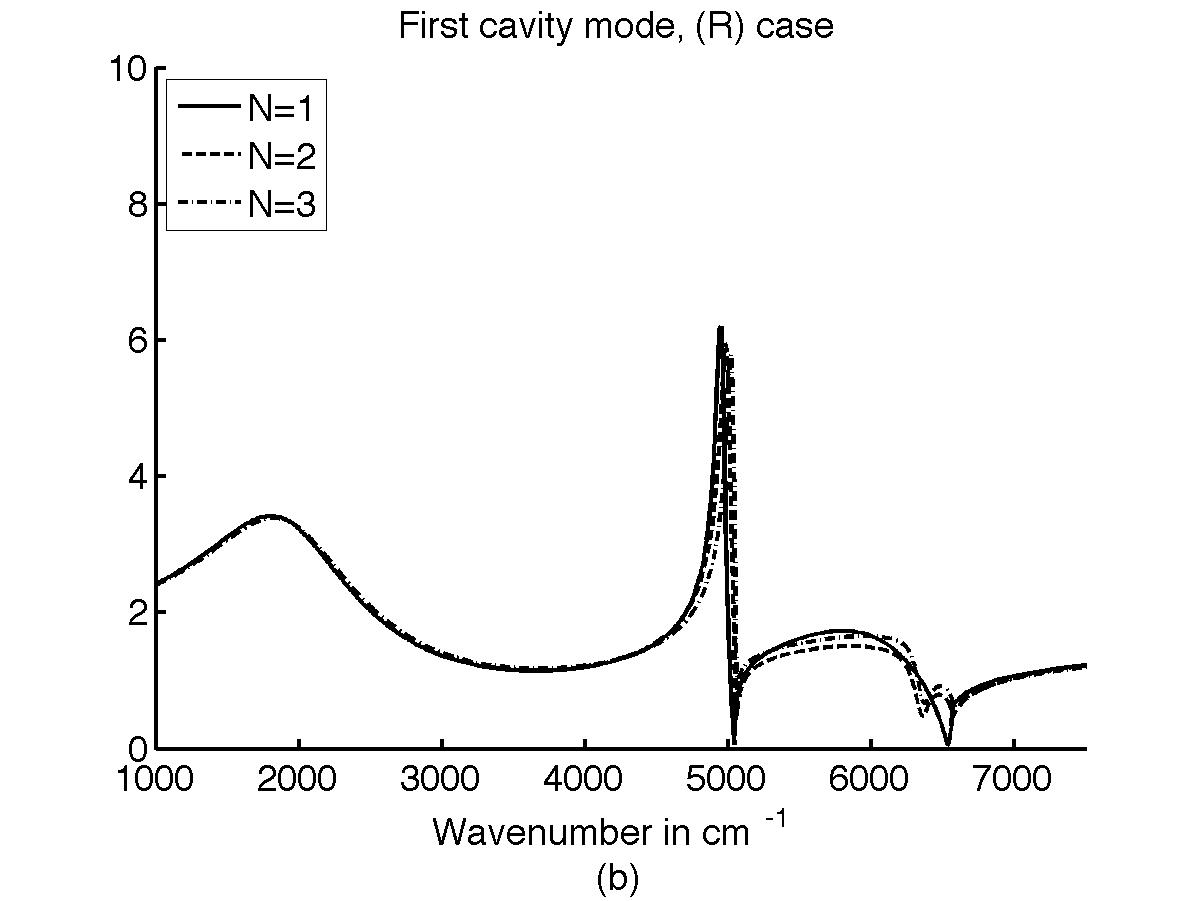}
\end{minipage}
\caption{\label{fig.A_0}Amplitude of the first renormalized cavity mode $A_0$ in the (M) case (a), and the (R) case (b) when computing $N=1$, 2 or 3 cavity modes.}
\end{center}
\end{figure}

The overall aspect is roughly always the same. 
When computing 1 mode, the (M) and (R) cases yield dips and bumps at about the same wavenumbers.
The impact of taking 2 or 3 cavity modes is low. 
The peak around $5000\cminv$ is very slightly shifted to the right.
A difference between the (M) and (R) cases is that the dip around $6500 \cminv$ disappears in the (R) case.
This certainly is connected with the same difference in the curves at those frequencies pointed out on Figure \ref{fig.N}. 
Since the (R) case was then closer to the experimental reflectance results, we can expect that the cavity modes displayed on figure \ref{fig.A_0}(b) for 2 and 3 cavity modes are closer to reality.

Now let us investigate the next two modes. 
We plot on Figure \ref{fig.A_n} only the results for $N=3$ preferring to compare the (M) and the (R) cases. 
 
\begin{figure}[htbp]
\begin{center}
\begin{minipage}{.45\linewidth}
\includegraphics[width=\textwidth]{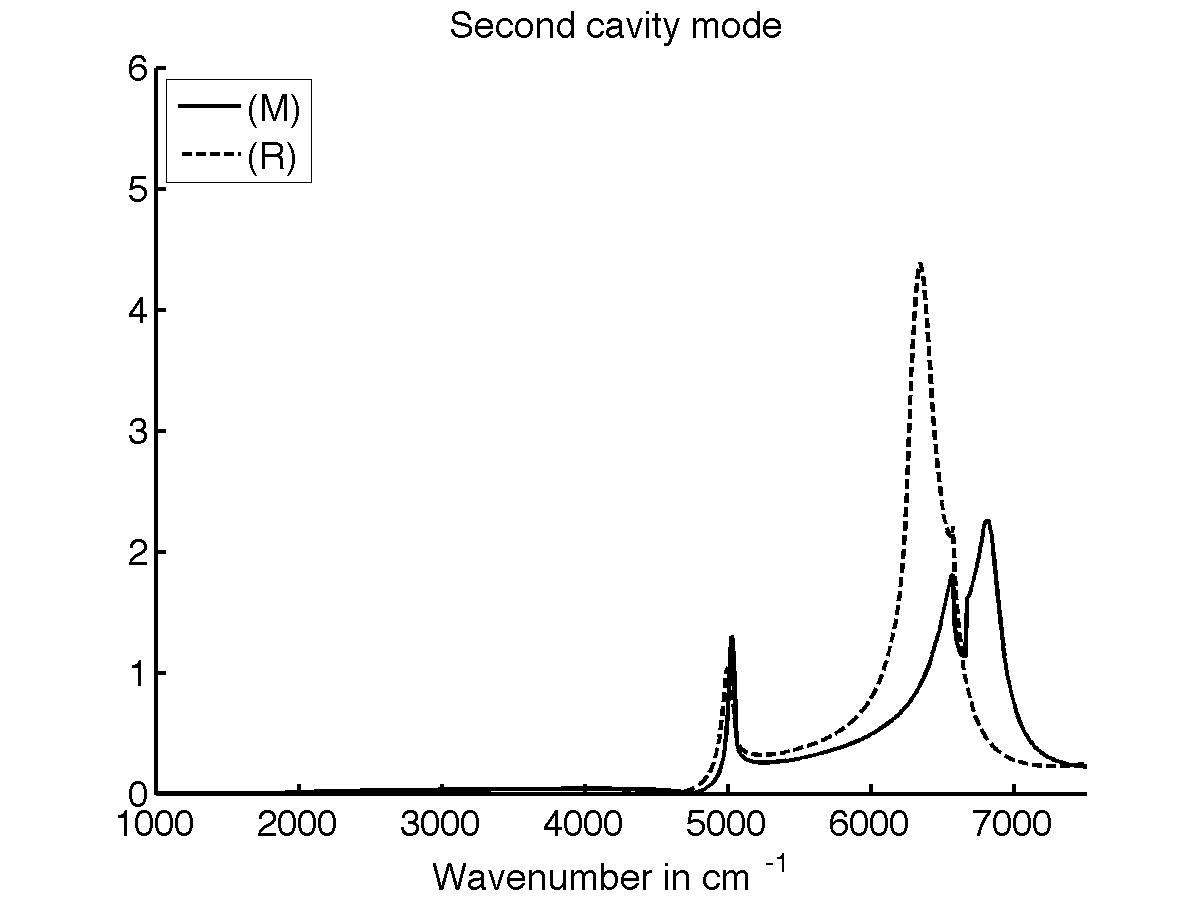}\\
\end{minipage}
\begin{minipage}{.45\linewidth}
\includegraphics[width=\textwidth]{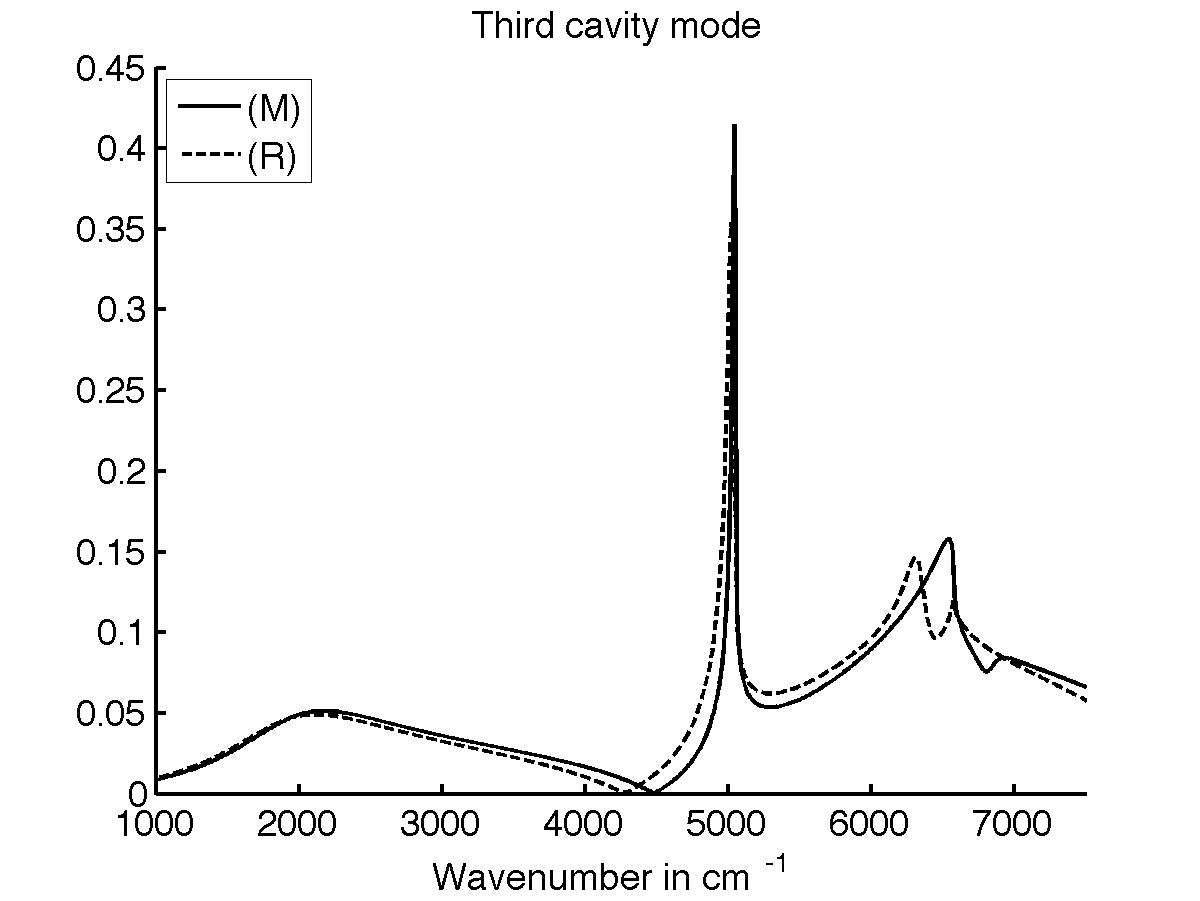}\\
\end{minipage}
\caption{\label{fig.A_n}Amplitude of the (a) second $a_1$ and (b) third $a_2$ renormalized cavity modes in the (M) and (R) cases.}
\end{center}
\end{figure}

As could be predicted from Figure \ref{fig.energy_pos}, the phenomena occur for smaller wavenumbers in the (R) case than in the (M) case.
They appear even sooner as Figure \ref{fig.energy_pos} would induce, which means that we should take into account the first evanescent modes even in the range where the monomode assumption is supposed to be valid.
We see that the second cavity mode can be predominant over the fundamental mode for some wavenumber.
This means that we can "switch on" almost pure second cavity modes for precise experimental settings. 
 
\section{Conclusion}

We have investigated the role of the choice of boundary conditions on the description of reflected and cavity modes in sub-wavelength gratings. 
As one could expect we have shown that a better agreement with experiments can be achieved by considering real-metal conditions on all boundaries.
Using asymptotic expansions, we have found a mean to produce numerical simulations in this full metal case which are not much more intricate that the mixed case previously used in the literature.  
On the way we have studied the role of the numerical parameters used to solve the theoretical infinite dimensional systems, the need to keep a certain number of evanescent modes, and a way to perform computations without being perturbed by conditioning issues.
 
\section*{Acknowledgments}
This project has been supported by the grant \textsc{madison} from the Universit\'e Joseph Fourier.
The authors wish to thank Aude Barbara and Pascal Qu\'emerais for fruitful discussions.


\end{document}